\documentclass[a4paper,1pt]{article}

\RequirePackage[OT1]{fontenc}
\RequirePackage{amsthm,amsmath}
\RequirePackage[numbers]{natbib}
\RequirePackage[colorlinks,citecolor=blue,urlcolor=blue]{hyperref}

\usepackage{mathrsfs}
\usepackage{amsmath, amssymb, amsthm, amsfonts, cases}

\allowdisplaybreaks

\newcommand{\X}{ {\mathcal{X}}}
\newcommand{\T}{ {\mathcal{T}}}
\newcommand{\U}{ {\mathcal{U}}}
\newcommand{\A} { {A}}
\newcommand{\B}{ {B}}
\newcommand{\C}{ {C}}
\newcommand{\D}{ {D}}

\newcommand{\LL} {\mathscr{L^*}} 
\newcommand{\GG} {\mathscr{G^*}} 

\newcommand{\LLinf} {\mathscr{L}} 
\newcommand{\GGinf} {\mathscr{G}}

\newcommand{\ExpArgScnd}{{2\sqrt{\A}t}}

\newcommand{\Ffirst}{\mathscr{F}_1}
\newcommand{\Fsecond}{\mathscr{F}_2}
\newcommand{\Fthird}{\mathscr{F}_3}
\newcommand{\Ffourth}{\mathscr{F}_4}
\newcommand{\Fn}{\mathscr{F}_n}

\newcommand{\THfirst}{\theta_{\Ffirst}}
\newcommand{\THsecond}{\theta_{\Fsecond}}
\newcommand{\THthird}{\theta_{\Fthird}}
\newcommand{\THfourth}{\theta_{\Ffourth}}
\newcommand{\THn}{\theta_{\mathscr{F}_n}}

\newcommand{\Vfield}{\textbf{v}}
\newcommand{\Sym}{\textbf{s}}

\numberwithin{equation}{section}
\theoremstyle{plain}
\newtheorem{thm}{Theorem}[section]

\newtheorem{example}{Example}[section]
\newtheorem{proposition}{Proposition}[section]

\newtheorem{remark}{Remark}[section]

\begin{document}

\title{Lie symmetries methods in boundary crossing problems for diffusion processes}
\author{Dmitry Muravey\footnote{e-mail:d.muravey87@gmail.com.}}

\date{}
\maketitle

\begin{abstract}
This paper uses Lie symmetry methods to analyze boundary crossing probabilities for a large class of diffusion processes. We show that if Fokker--Planck--Kolmogorov equation has non-trivial Lie symmetry, then boundary crossing identity exists and depends only on parameters of process and symmetry. For time-homogeneous diffusion processes we found necessary and sufficient conditions of symmetries' existence. This paper shows that if drift function satisfy one of a family of Ricatti equations, then the problem has nontrivial Lie symmetries. For each case we present symmetries in explicit form. Based on obtained results, we derive two-parametric boundary crossing identities and prove its uniqueness. Further, we present boundary crossing identities between different process. We show, that if the problem has 6 or 4 group of symmetries then the first passage time density to any boundary can be explicitly represented in terms of the first passage time by Brownian motion or Bessel process. Many examples are presented to illustrate the method.    
\end{abstract}

\section{Introduction}
Let $b(t)$ be a sufficiently smooth function of time, $W = \left \{W_t : t \geq 0 \right\}$ be a standard Brownian motion and $X = \left\{X_t : t \geq 0 \right\}$ be a solution to the following SDE
\begin{equation}
\label{eq:SDE X}
dX_t = \mu(X_t, t) dt + \sigma(X_t, t) dW_t, \quad X_0 = x_0 \leq b(0).
\end{equation}
We assume that functions $\mu(x,t)$ and $\sigma(x,t)$ are smooth functions such that equation (\ref{eq:SDE X}) has a unique non-explosive solution for any initial value $x_0$. We define the first passage time $T^b$ of the diffusion process $X_t$ to the curved boundary $b(t)$:   
\begin{equation}
\label{eq:tau_definition}
T^b = inf\{t >0;\, X_t \leq b(t)\}.
\end{equation}
We denote by $\mathbb{P}_{x_0}$ and $\mathbb{E}_{x_0}$ probability and expectation conditional on the process $X_t$ started at $X_0 = x_0$. We omit subscript $x_0$ if process $X_t$ starts at zero, i.e. $x_0 = 0$. Define function $u_{x_0}^b(x, t)$ as 
\begin{equation}
\label{eq: u definition}
u_{x_0}^b(x, t) = \frac{\partial } {\partial x} \mathbb{P}_{x_0} \left( X_t \leq x, \, T^b > t\right). 
\end{equation}
From standard results in probability theory the function $u_{x_0}^b(x,t)$ solves Cauchy problem with absorbing boundary condition for Fokker--Planck--Kolmogorov equation (FPKE)
\begin{eqnarray}
\label{eq:PDE_for_u}
\left\{ {\begin{array}{l}
\left( \LL - \frac{\partial }{\partial t}\right) u_{x_0}^b = 0, \quad (x,t) \in \left( b(t), \infty\right) \times \left(0, \infty\right)
\\  
u_{x_0}^b (b(t),t) = 0, \quad t \in \left[0, \infty \right)
\\ 
u_{x_0}^b(x, 0) = \delta(x - x_0), \quad x \in \left[ b(0), \infty \right) 
\end{array}} \right.
\end{eqnarray}
where $\delta(x-x_0)$ is a Dirac measure at the point $x_0$ and differential operator $\LL$ is the adjoint of $\LLinf$, which is the infinitesimal generator of $X_t$:
\begin{equation}
\label{eq: L def}
\LLinf = \frac{\sigma^2(x,t)}{2} \frac{\partial ^2}{\partial x^2 }+ 
\mu(x,t)\frac{\partial}{\partial x } ,\quad \LL= \frac{1}{2} \frac{\partial ^2}{\partial x^2 } \sigma^2(x,t) -
\frac{\partial}{\partial x } \mu(x,t).
\end{equation}
Superscript $b$ in (\ref{eq: u definition}) denotes boundary $b(t)$. If $b(t)$ is sufficiently regular, the problem (\ref{eq:PDE_for_u}) has unique solution and probability $\mathbb{P}_{x_0} (T^b > t)$ and density $\rho_{x_0}^{T^b}(t)$ of the moment $T^b$ have the following representations (we also omit subscript $x_0$ in case of $x_0 = 0$): 
\begin{equation}
\label{eq:P_tau and rho _tau def} 
\begin{array}{l}
\mathbb{P}_{x_0} (T^b > t) = \int_{b(t)}^\infty u_{x_0}^b (x, t) dx,
\\ 
\rho^{T^b}_{x_0} (t) = \frac{\partial}{\partial x} \left( \frac{\sigma^2(x,t) u_{x_0}^b(x,t)}{2} \right) |_{x = b(t)}.
\end{array}
\end{equation}
The calculation of boundary-crossing probabilities (\ref{eq:P_tau and rho _tau def}) arises in several areas such as statistical testing \cite{Ferebee, RobbinsAndSiegmund}, the valuation of barrier options \cite{GemanYor, KunimotoIkeda, RobertsShortland}, default modeling \cite{Coculescu, HullWhite},  molecular biology \cite{CharmpiYcart}. The comprehensive review of the first passage time problems can be found in \cite{FPT_Phenomena}. 

The most popular analytic techniques to analyze first passage times are martingale theory via optional sampling theorem, measure/time changes, integral equations and PDE theory. Martingale methods were used in \cite{BorovkovDownes2010, Novikov Korzakhia, Novikov Stopping times for Brownian motion}. Method of images and integral transforms are the most frequently used techniques of PDE theory. The method of images is used in \cite{Daniels, Kahale, Lerche} for Brownian Motion and in \cite{HittimeBessel} for Bessel process. If we study first passage times to a fixed level by a time-homogeneous process we can apply Laplace transform and reduce original PDE problem to the ODE problem. This technique gives the first passage time law in the series expansion form, for details see \cite{Linetsky}. Paper \cite{Peskir} shows, that first hitting time density can be found as a solution of Volterra integral equation. 

In 2005 Alili and Patie have published a paper \cite{AliliPatie05} about so-called functional transformation approach for the first passage time problems. They considered a one-parametric mapping from the space of continuous positive functions into a family of curves and established explicit formula relating the distributions of the first hitting times of each of these curves by a Brownian Motion. Later, they generalized their results to the processes having time-inversion property \cite{AliliPatie10}. In the recent paper \cite{AliliPatieLastWork} they extended one-parametric transformation to the two-parametric and proved uniqueness of proposed transform. Their proof is based on Lie symmetries for the heat equation. 

In this article we explore connections between Lie symmetries of FPKE and laws of first passage times of a curved boundary by a diffusion process. In the assumption of non-trivial Lie symmetry's existence we derive explicit boundary crossing identity characterized only by symmetry and process parameters. This result generalizes Alili and Patie transformations to the case of diffusion process such that FPKE has non-trivial Lie symmetries. Also we generalize obtained identity to the case of two different diffusion process such that their FPKE have a one specific relation similar to the Lie symmetry. The main contribution of this paper is these two explicit identities. 

Further, we focus on time-homogeneous diffusion processes. Without loss of generality we consider process with unit variance. For these processes we find necessary and sufficient conditions of symmetries' existence and derive all of Lie symmetries in explicit form.  Based on these results, we obtain two-parametric transforms and prove its uniqueness. We apply reductions to the canonical form (see \cite{Goard}) and obtain boundary crossing identities between time-homogeneous diffusion process and Brownian motion or Bessel Processes. These results generalize well-known connections between Ornstein-- Uhlenbeck process and Brownian Motion and radial Ornstein--Uhlenbeck process and Bessel process. 

The rest of paper is organized this way: in Section \ref{sec:main_results} we briefly introduce Lie symmetries method and present boundary crossing identities for general case. In Section \ref{sec:time_homogeneous_diffusion} we focus on time homogeneous diffusion processes. We show that existence of non-trivial Lie symmetry is equivalent to the condition that drift function solves one of 4 Ricatti equations. Each Ricatti equation can be solved explicitly in terms of elementary and special functions. For each family we present closed form solution of corresponded Ricatti equation, derive all symmetries in explicit form, construct two-parametric transformations and prove its uniqueness.  

\section{Boundary crossing identities and Lie symmetries. General results}
\label{sec:main_results}
\subsection{Preliminaries of Lie symmetries}
A Lie group symmetry of a differential equation is a transformation which maps solutions to solutions. Application of this technique allows us to construct complex solutions from trivial solutions. Here we briefly review the main idea of Lie method, more details can be found in book by Olver \cite{Olver}.  Let $u(x,t)$ be a solution of equation $\left(\LL - \partial_t \right) u = 0$ and $\textbf{v}$ be a vector field of the form 
\begin{equation}
\Vfield = \xi(x,t,u)\partial_x + \tau(x,t,u)\partial_t + \varphi(x,t,u)\partial_u 
\end{equation}
where $\partial_x$ denotes partial derivative $\partial / \partial_x$ etc. Functions $\xi$, $\tau$ and $\varphi$ are sought to guarantee that field $\Vfield$ generates a symmetry of PDE $\left(\LL - \partial_t \right) u = 0$. According to the Lie's theorem, vector field $\Vfield$ generates local group of symmetries if and only if its second prolongation (see Chapter 2 of \cite{Olver} for the explicit formulas of prolongations) is equal to zero:  
\begin{equation}
\label{eq:second_prolongation}
\textbf{pr}^2\Vfield\left[ \left(\LL - \partial_t \right) u  \right] = 0.
\end{equation}
Action of the vector field $\Vfield$ gives triplet $(\X, \T, \U)$ (see Chapter 1, \cite{Olver})
\begin{equation}
\label{eq:exponentiation}
e^{\epsilon \Vfield} (x,t,u) = (\X,\T,\U)
\end{equation} 
such that $(\X, \T, \U)$ solves the following ODE system
\begin{eqnarray}
\label{eq: systems to determine action of vector field}
\left\{ {\begin{array}{l}
d \X /d\epsilon = \xi(\X,\T,\U), 
\quad \X|_{\epsilon = 0} = x, 
\\
d \T / d\epsilon = \tau(\X, \T, \U), 
\quad \T|_{\epsilon = 0} = t, 
\\  
d \U / d\epsilon = \varphi(\X,\T, \U), 
\quad \U|_{\epsilon = 0} = u. 
\end{array}} \right.
\end{eqnarray}
Therefore, condition that vector field $\Vfield$ generates symmetry of differential operator $\LL - \partial_t$ is equivalent to existence of functions $f(x,t)$, $\X(x,t)$ and $\T(t)$ such that for any solution $\left(\LL - \partial_\T \right) \U =0$ the function $u(x,t)$ defined as 
\begin{eqnarray}
\label{eq:Basic_transform}
u(x,t) = f(x, t) \U \left(\X(x,t), \T(t) \right),
\end{eqnarray}
solves the equation $\left(\LL - \partial_t \right) u =0$. In order to obtain symmetries (\ref{eq:Basic_transform}) we need to solve the equation for second prolongation (\ref{eq:second_prolongation}) and after that make exponentiation (\ref{eq:exponentiation}). This steps gives us all symmetries (\ref{eq:Basic_transform}).        

\subsection{Boundary crossing identities}
In this subsection we present first main result of this paper. It is explicit identity between densities of first passage times to the different boundaries.
\label{par: main result}
\begin{thm}
	\label{thm:main}
	Suppose that equation $\left( \LL - \partial_t \right) u = 0$ has non-trivial Lie symmetry (\ref{eq:Basic_transform}) such that $\T(0) = 0$. Let $T^b$ be the first passage time of the curve $b(t)$ by the process $X_t$: 
	\begin{equation}
	T^b = \inf \left\{  t \geq 0, \, X_t \leq b(t) \right\}.
	\end{equation} 
	Let's define by $g(t)$ any solution of equation 
	\begin{equation}
		\label{eq:g_def}
		\X(g(t),t) =  b(\T).
	\end{equation}
	and by $T^g$ the corresponded stopping time
	\begin{equation}
	\label{eq:FPT_g}
		T^g = \inf \left\{  t \geq 0, \, X_t \leq g(t) \right\}.
	\end{equation} 
	The density $\rho^{T^g}_{x_0} (t)$ of the distribution of $T^g$ can be explicitly represented from the following identity: 
	\begin{equation}
	\label{eq:rho_transform}
	\frac{\rho^{T^g}_{x_0} (t)}
	     {\rho^{T^b}_{\X(x_0, 0)} (\T(t))} = 
	\frac{\sigma^2(g(t), t)}{\sigma^2(b(\T(t)),\T(t))} 
	\frac{f(g(t),t)}{f(x_0, 0)}  
	\frac{\partial \X / \partial x |_{x = g(t)}}{\Psi'(\X(x_0, 0))} .
	\end{equation} 
	Here function $\Psi(\mathscr{Y})$ is defined as a solution of the following equation
	\begin{equation}
	\label{eq:Psi_def}
		\X(\Psi(\mathscr{Y}),0) = \mathscr{Y}.
	\end{equation}
\end{thm}
\begin{proof}
	Let function $\U(\X,\T)$ be a solution of Cauchy problem 
	\begin{eqnarray}
	%	\label{eq:PDE_for_U}
	\left\{ \begin{array}{l}
	%\LL_{x, t} u = \frac{1}{2} \left(\sigma^2 u\right)_{xx} - (\mu u)_x -u_t = 0, \\
    \left( \LL -\partial_{\T} \right) \U = 0, 
    \quad \left(\X, \T \right) \in \left(b(\T), \infty \right) \times \left(0,\infty\right)
	\\ 
	\U(b(\T),\T) = 0, 
	\quad \T  \in \left[0,\infty\right)
	\\ 
	\U(\X, 0) = \frac{\delta(\Psi(\X) - x_0)}{f(\Psi(\X), 0)}, 
	\quad \X \in \left[b(0), \infty\right). 
		\end{array} \right.
	\end{eqnarray}
	It follows that the function $u^g_{x_0}(x, t) = f(x,t) \U(\X,\T)$ solves the problem 
	\begin{eqnarray}
	\label{eq:PDE_for_u_g}
	\left\{ {\begin{array}{l}
	\left( \LL - \partial_t \right) u^g_{x_0}  = 0,  
    \quad \left(x, t\right) \in \left(g(t), \infty \right) \times \left(0, \infty\right)
	\\ 
	u^g_{x_0}(g(t),t) = 0, 
	\quad t  \in \left[0,\infty\right)
	\\ 
	u^g_{x_0}(x, 0) = \delta(x - x_0). 
	\quad x \in \left[g(0), \infty\right). 
		\end{array}} \right.
	\end{eqnarray}
	Let function $u^b_{x_0}(x, t)$ be a solution to the equation (\ref{eq:PDE_for_u}). Therefore, we can represent function $\U$ as the following convolution product:
	\begin{eqnarray}
	\nonumber
	\U \left(\X, \T \right) &=& \int_{b(0)}^{\infty} u_{\X_0}^b(\X, \T) \U(\X_0, 0) d\X_0
	\\ \nonumber
	&=& \int_{b(0)}^{\infty} \frac{u_{\X_0}^b(\X, \T) \delta(\Psi(\X_0) - x_0)}{f(\Psi(\X_0), 0)} d\X_0
	\\\nonumber
	& = & \int_{b(0)}^{\infty} \frac{u_{\X_0}^b(\X, \T) \delta(\Psi(\X_0) - x_0)}{\Psi'(X_0)f(\Psi(\X_0), 0)} d\Psi(\X_0)
	\\\nonumber
	& = & \frac{u_{\X_0}^b(\X, \T)}{\Psi'(X_0)f(\Psi(\X_0), 0)}|_{\Psi(\X_0) = x_0}
	\\ \nonumber
	&=& \frac{u_{\X(x_0,0)}^b (\X, \T)}{\Psi'(\X)|_{\X(x_0,0)} f (x_0,0)}.
	\end{eqnarray}
	After the following calculations, we get the main identity (\ref{eq:rho_transform}):
	\begin{eqnarray}
	\nonumber
	\rho^{T^g}_{x_0} (t) 
	&=& 
	\frac{1}{2} \frac{\partial}{\partial x} \left( \sigma^2(x,t) u^g_{x_0}(x,t) \right) |_{x = g(t)} 
	\\ \nonumber
	&=& 
	\frac{1}{2} \frac{\partial}{\partial x} 
	\left(\frac{\sigma^2(x,t) u^b_{\X(x_0,0)}(\X, \T) f(x,t)}{\Psi'(\X)|_{\X = \X(x_0,0)} f (x_0,0)}\right) |_{x = g(t)} 
	\\ \nonumber
	&=& 
	\frac{1}{2} \frac{\partial}{\partial x} 
	\left(\frac{\sigma^2(x,t)  \sigma^2(\X,\T)  u^b_{\X(x_0,0)}(\X, \T) f(x,t)}{\sigma^2(\X, \T) \Psi'(\X)|_{\X(x_0,0)} f (x_0,0)}\right) |_{x = g(t)}
	\\ \nonumber
	&=& 
	\frac{1}{2} \sigma^2(\X, \T) u^b_{\X(x_0,0)}(\X, \T)|_{x = g(t)} 
	\\ \nonumber
	&\cdot& 
	\frac{\partial}{\partial x} 
	\left(\frac{\sigma^2(x,t) f(x,t)}{\sigma^2(\X, \T) \Psi'(\X)|_{ \X = \X(x_0,0)} f (x_0,0)}\right) |_{x = g(t)}
	\\ \nonumber
	&+&
	\frac{1}{2} \left(\frac{\sigma^2(x,t) f(x,t)}{\sigma^2(\X, \T) \Psi'(\X)|_{\X(x_0,0)} f (x_0,0)}\right) |_{x = g(t)}
	\\ \nonumber
	&\cdot&
	\frac{\partial}{\partial x}  
	\left( \sigma^2(\X,\T) u^b_{\X(x_0,0)}(\X, \T)\right)|_{x = g(t)} 
	\\ \nonumber
	&=& \frac{1}{2} \sigma^2(\X, \T) u^b_{\X(x_0,0)}(\X, \T) |_{\X = b(\T)} 
	\\ \nonumber
	&\cdot&
	\frac{\partial}{\partial x} 
	\left(\frac{\sigma^2(x,t) f(x,t)}{\sigma^2(\X, \T) \Psi'(\X)|_{\X(x_0,0)} f (x_0,0)}\right) |_{x = g(t)} 
	\\ \nonumber
	&+&\frac{1}{2} \left(\frac{\sigma^2(x,t) f(x,t) \partial \X / \partial x}{\sigma^2(\X, \T) \Psi'(\X)|_{\X(x_0,0)} f (x_0,0)}\right) |_{x = g(t)}
	\\ \nonumber
	&\cdot&
	\frac{\partial}{\partial \X}  
	\left( \sigma^2(\X,\T) u^b_{\X(x_0,0)}(\X, \T)\right)|_{\X = b(\T)} 
	\\ \nonumber
	&=& 0 + \left(\frac{\sigma^2(x,t) f(x,t) \partial \X / \partial x}{\sigma^2(\X, \T) \Psi'(\X)|_{\X(x_0,0)} f (x_0,0)}\right) |_{x = g(t)} \rho^{T^b}_{\X(x_0,0)} (\T(t)) 
	\\ \nonumber
	&=& \frac{\sigma^2(g(t), t)}{\sigma^2(b(\T(t)),\T(t))} \frac{f(g(t),t)}{f(x_0, 0)}  \frac{\partial \X / \partial x |_{x = g(t)}}{\Psi'(\X)|_{\X(x_0, 0)}} \rho^{T^b}_{\X(x_0, 0)} (\T(t)).
	\end{eqnarray}
\end{proof}

\begin{example} We illustrate an application of the main formula (\ref{eq:rho_transform}) 
with the well-known Bachelier--L$\acute{e}$vy formula for the first passage time of a Brownian motion to the slopping line $g(t) = a + b t$. We present two transformations of the density corresponds to the straight line, i.e. $b(t) = a$,  $\rho^{T^b} = \frac{|a|}{\sqrt{2\pi t^3}} e^{-\frac{a^2}{2t}}$. Recall all one-parametric symmetry for the heat equation(see Chapter 2 in \cite{Olver}):  
\begin{eqnarray}
\label{eq:symmetries_for_the_heat_eq}
\left. { \begin{array}{l}
\Sym^{\epsilon}_1: \quad u(x,t) = \frac{1}{\sqrt{1+ \epsilon t}}e^{\frac{-\epsilon  x^2}{2(1+ \epsilon t)}}  \U\left(\frac{x}{1+\epsilon t }, \frac{t}{1+\epsilon t } \right)  ,  
\\ 
\Sym^{\epsilon}_2: \quad u(x,t) =  \U(e^{\epsilon}x, e^{2\epsilon} t), 
\\ 
\Sym^{\epsilon}_3: \quad u(x,t) = e^{-\epsilon x + \epsilon^2 t / 2}\U(x -\epsilon t, t), 
\\ 
\Sym^{\epsilon}_4: \quad u(x,t) = \U(x+\epsilon, t) 
\\ 
\Sym^{\epsilon}_5: \quad u(x,t) = \U(x, t+ \epsilon), 
\\ 
\Sym^{\epsilon}_6: \quad u(x,t) = e^\epsilon \U(x, t), 
\\ 
\Sym_\beta: \quad u(x,t) =  \beta(x,t) + \U\left(x, t \right), 	
	\end{array}} \right.
\end{eqnarray}
where function $\beta(x,t)$ is an arbitrary solution of equation $\left(\LL - \partial_t \right)\beta = 0$.  The probabilistic interpretation of any symmetry from (\ref{eq:symmetries_for_the_heat_eq}) can be found in \cite{AliliPatieLastWork}.  Application of symmetry $\Sym^{\epsilon}_4$ with $\epsilon = b$ yields
\begin{equation}
u(x,t) = e^{-a x + b ^ 2 t / 2} \U(x - a t, t). 
\end{equation}
New boundary $g(t)$ solves the equation $\X(g(t), t) = g(t) - bt = a$. If we apply formula (\ref{eq:rho_transform}) we get ($f(0, 0) = 1$, $d \X /dx \equiv 1$, $\Psi' \equiv 1$)   
\begin{eqnarray}
\label{eq:Bachelier_Levy_formula}
\rho^{T^g} (t) &=& \frac{|a}{\sqrt{2\pi t^3}} e^{-\frac{a^2}{2t}}
e^{-b (a + b t) + b ^ 2 t / 2} 
\\ \nonumber
&=& \frac{|a|}{\sqrt{2\pi t^3}} e^{-\frac{a^2}{2t} -ab- \frac{b ^ 2 t}{2}}. 
\end{eqnarray}
Now we consider the superposition of symmetries $\Sym^{b/a}_5 \circ \Sym^{b/a}_1$:
\begin{equation}
u(x,t) =  
\frac{e^{-\frac{x^2}{2(t + a / b) }}}{\sqrt{t + a/ b}} 
\U \left(\frac{x}{t + a / b}, \frac{t} {(t + a /b) a / b }\right). 
\end{equation}
We apply this symmetry to the boundary $b(t) = b$ ($\rho^{T^b} = \frac{|b|}{\sqrt{2\pi t^3}} e^{-\frac{b^2}{2t}}$). In the result we get $g(t) = (t + a / b) b(\T(t)) = a + b t$.  For the density $\rho^{T^g}(t)$ we have
\begin{eqnarray}
\rho^{T^g} (t) &=& \frac{\sqrt{b / a}}{\sqrt{(t + a/ b)^3}}  e^ {-\frac{ b^2 t + ab}{2}}\rho^{T^b}\left(\frac{t} {(t + a/b) a / b }\right) 
\\ \nonumber
&=& \frac{|a|}{\sqrt{2\pi t^3}} e^{-\frac{a^2}{2t} - a b - \frac{b ^ 2 t}{2}} . 
\end{eqnarray}
\end{example}

\begin{remark}
Main transform from \cite{AliliPatie05} can be easily derived by the application of Theorem \ref{thm:main} with symmetry $\Sym^{\beta}_{1}$ from (\ref{eq:symmetries_for_the_heat_eq}). If we apply Theorem \ref{thm:main} with symmetry $\Sym^{\epsilon}_{4}$ we obtain first hitting time laws by Brownian motion with drift (i.e. Girsanov transform).  
\end{remark}

\begin{remark}
	In this paper we apply identity (\ref{eq:Basic_transform}) to derive closed form formulas for the first passage time laws. However, obtained identity can be useful in numerical context. We can use it as a benchmark of any numerical methods.
\end{remark}

\subsection{Boundary crossing identities for different processes}
In this section we introduce another diffusion process $Y = \left\{Y_t, \,t \geq 0 \right\}$ defined by SDE 
\begin{equation}
dY_t = \nu \left(Y_t, t \right) dt + \Sigma \left( Y_t, t \right) dW_t, \quad Y_0 = y_0,
\end{equation} 
Let $T^b$ and $T^g$ also be the following stopping times of the processes $Y_t$ and $X_t$ started from $\X(x_0, 0)$ and $x_0$ respectively 
\begin{equation}
\label{eq:tau_b_and_tau_g_diff_process}
T^b = \inf \left\{  t \geq 0, \quad Y_t \leq b(t) \right\}, \quad 
T^g = \inf \left\{  t \geq 0, \quad X_t \leq g(t) \right\}.
\end{equation}
Next Theorem provides explicit connection formulas between distributions of $T_Y^b$ and $T^g$. 
\begin{thm}
	\label{thm:main_ch_mes}
	Suppose that solutions of equations $\left(\LL - \partial_t\right) u = 0$  and  $\left(\GG - \partial_\T\right) \U = 0$ have relation (\ref{eq:Basic_transform}). Here $\GGinf$ is infinitesimal generator of the process $Y_t$. Let $b(t)$ and $g(t)$ be smooth functions having relation (\ref{eq:g_def}). Then, the boundary crossing identity between densities of stopping moments (\ref{eq:tau_b_and_tau_g_diff_process}) can be represented in the form (\ref{eq:rho_transform}). 
\end{thm}
\begin{proof}
	Proof is fully analogical to the proof of Theorem \ref{thm:main}.
\end{proof}

\begin{example}
We also illustrate an application of Theorem \ref{thm:main_ch_mes} with Bachelier--L$\acute{e}$vy formula. Let $X_t$ be a Brownian motion with negative drift $-b$ and $Y_t$ be a standard Brownian Motion, i.e
\begin{equation}
\LL = \frac{1}{2} \frac{\partial^2}{\partial x^2} + b \frac{\partial }{\partial x}, \quad 
\GG = \frac{1}{2} \frac{\partial^2}{\partial x^2}.  
\end{equation}
After application of formula (\ref{eq:rho_transform}) with $\X \equiv x$, $\T \equiv t$, $f(x,t) = e^{-b x - b^2 t  /2}$, $b(t) \equiv g(t) \equiv a$ and $x_0 = 0$ we get formula (\ref{eq:Bachelier_Levy_formula}).
%\begin{equation}
%\rho^{T^g}_{x_0} (t) = e^{-b a - b^2 t / 2 + b x_0} \rho^{T^b}_{x_0} (\T(t)).
%\end{equation}
%\begin{equation}
%\rho^{T^g} (t; \, X_t) = \frac{|a|}{\sqrt{2\pi t^3}} e^{-\frac{ a^2}{2 t}- a b - \frac{b^2 t} {2}}.
%\end{equation}
\end{example}

\section{Time-homogeneous diffusion processes}
\label{sec:time_homogeneous_diffusion}
Now we turn to the time-homogeneous diffusion processes, i.e. $\mu(x, t) \equiv \mu(x)$ and $\sigma(x, t) \equiv \sigma(x)$. Without loss of generality, we assume that $\sigma(x) \equiv 1$. In the case of $\sigma(x) \neq 1$  we can apply Lamberti transform (see e.g. \cite{BorovkovDownes2010}): 
\begin{eqnarray}
\label{eq:Lamperti_transform}	 
\left. { \begin{array}{l}
dU_t = \nu(U_t) dt + \sigma(U_t) dW_t, \quad
F(y) = \int_{y_0}^y \frac{du}{\sigma(u)}, \quad 
X_t = F(U_t); 
\\
dX_t = \mu(X_t) dt + dW_t, 
\quad \mu(y) = \left( \frac{\nu}{\sigma} - \frac{\sigma'}{2} \right) \circ F^{-1}(y).
\end{array}} \right.
\end{eqnarray}
Differential operator $\LL$ is simplified to
\begin{equation}
\label{eq:time_homogeneous_L}
\LL = \frac{1}{2} \frac{\partial^2}{\partial^2 x} - 
\mu(x) \frac{\partial}{\partial x} - \mu'(x). 
\end{equation}
Next Theorem provides necessary and sufficient conditions of non-trivial symmetry's existence. 
\begin{thm}
	\label{thm:Lie_symm_existence}
	Differential Equation $(\LL  - \partial_t)u = 0$ where $\LL$ defined in (\ref{eq:time_homogeneous_L}) has non-trivial Lie symmetry (\ref{eq:Basic_transform}) if and only if drift function $\mu(x)$ is a solution of one of the following Ricatti equations:
	\begin{eqnarray}
	\label{eq:Ricatti_equations}
	\left. { \begin{array}{l}
	\mathscr{F}_1 \, : \,  \mu' + \mu^2 = 4 \B x  + 2\C 
	\\ 
	\mathscr{F}_2 \, : \,  \mu' + \mu^2 = \A x^2 + 4 \B x  + 2\C 
	\\ 
	\mathscr{F}_3 \, : \,  \mu' + \mu^2 = 2\C + \D / x^2 
	\\ 
	\mathscr{F}_4 \, : \,  \mu' + \mu^2 = \A x^2 + 2\C + \D / x^2
\end{array} } \right.
	\end{eqnarray} 		
	where $\A$, $\B$, $\C$ and $\D$ are arbitrary constants.
\end{thm}
\begin{proof}	
	Craddock and Lennox considered in \cite{CraddockLennox} Lie symmetries for the equation
	\begin{equation}
	\label{eq:craddock_lennox_eq}
	\sigma x^\gamma \frac{\partial^2 u}{\partial x} + f(x) \frac{\partial u}{\partial x} - g(x)u - \frac{\partial u}{\partial t} = 0. 
	\end{equation}
	They proved that if $f(x)$ solves Ricatti equation with coefficients depending on function $g$, then the operator (\ref{eq:craddock_lennox_eq}) has non-trivial Lie symmetries. Differential operator $\LL - \partial_t$ is the special case of (\ref{eq:craddock_lennox_eq}) with specifications $\gamma = 0$, $\sigma = 1/2$, $f(x) = -\mu(x)$ and $g(x) = \mu'(x)$. For readers' convenience we present proof from \cite{CraddockLennox} with our specifications.  According to (\cite[p.~129]{CraddockLennox}) components of vector field $\Vfield$ allow the following representation
	\begin{eqnarray}
	\left. { \begin{array}{c}
	\label{eq: field for time homogeneous}
	\xi(x,t) = \frac{x}{2} \tau_t + \rho, \quad\quad \varphi(x,t, u) = \alpha(x,t) u + \beta(x,t), 
	\\ 
	\alpha(x,t) = -\frac{x^2}{4} \tau_{tt} - x\rho_t +\mu(x) \rho + \frac{x}{2} \mu(x) \tau_t + \eta,
	\end{array} } \right.
	\end{eqnarray} 
	where $\tau$, $\rho$ and $\eta$ are functions of $t$ only and solve the equation
	\begin{equation}
	\label{eq: Base for time homogeneous case}
	-\frac{x^2}{4} \tau_{ttt} - x\rho_{tt} + \eta_t = -\frac{\tau_{tt}}{4} - L(x) \tau_t + K(x) \rho.
	\end{equation}
	Functions $K(x)$ and $L(x)$ are defined analogical to \cite[p.~129, formula 2.3]{CraddockLennox}
	\begin{equation}
	K(x) = -\frac{\left(\mu' + \mu^2 \right)'}{2}, \quad\quad 
	L(x) = \frac{x}{4} \left( \mu' + \mu^2 \right)' + \frac{\mu' +\mu^2}{2}.
	\end{equation}
	Equation (\ref{eq: Base for time homogeneous case}) has non-trivial solution if and only if function $L(x)$ is quadratic, i.e.
	\begin{equation}
	\frac{x}{4} \left( \mu' + \mu^2 \right)' + \frac{1}{2} \left( \mu' +\mu^2\right) = \A x^2 + 3 \B x + \C.
	\end{equation} 
	This ODE can be simply solved by standard techniques:
	\begin{eqnarray}
	\nonumber
	\mu' + \mu^2 &=& \frac{\D}{x^2} + \frac{1}{x^2} \int \left( 4 \A x^3 + 12 \B x^2 + 4\C x\right) dx
	\\ \nonumber
	&=& \A x^2 + 4 \B x + 2 \C + \frac{\D}{x^2}.
	\end{eqnarray}
	Functions $K(x)$ and $L(x)$ are equal to:
	\begin{equation}
	\label{eq: Explicit K and L}
	K(x) = -\left(\A x + 2 \B  - \frac{\D}{x^3} \right), \quad \quad L(x) = \A x^2 + 3 \B x + \C.
	\end{equation}
	Substitution of (\ref{eq: Explicit K and L}) in (\ref{eq: Base for time homogeneous case}) yields
	\begin{equation}
	\label{eq: Parameterized time Hom Base}
	-\frac{x^2}{4} \tau_{ttt} - x\rho_{tt} + \eta_t = -\frac{\tau_{tt}}{4} - \left(Ax^2 + 3 \B x + C \right) \tau_t - \left(Ax + 2 \B - \frac{D}{x^3} \right)\rho.
	\end{equation}
	Comparison of the corresponding powers turns out to the following ODE system
	\begin{eqnarray}
	\label{eq:main_system_for_symmetries}
	\left\{ {\begin{array}{l}
	\tau_{ttt} - 4 \A \tau_t = 0 
	\\ 
	\rho_{tt} - \A \rho =  3 \B \tau_t 
	\\ 
	\eta_t = -\left(\frac{1}{4}\tau_{tt} + \C \tau_t + 2 \B \rho \right) 
	\\ 
	\D \rho = 0
	\end{array}} \right.
	\end{eqnarray}
	It is easy to see that condition $\D \neq 0$ is equivalent to $\B=0$. Then, we have two families of Ricatti equations for $\mu$:
	\begin{eqnarray}
	\nonumber
	\begin{array}{l}		
	\mu' + \mu^2 = \A x^2 + 4\B x  + 2\C, 
	\\
	\mu' + \mu^2 = \A x^2 + 2\C + \frac{\D}{x^2}. 
	\end{array}
	\end{eqnarray} 		
	For convenience we arrange it into 4 families (\ref{eq:Ricatti_equations}).        
\end{proof}
\begin{remark}
It is easy to check that backward Kolmogorov equation for time-homogeneous process (\ref{eq:Lamperti_transform}) 
\[
\left(\LLinf + \frac{\partial }{\partial t} \right) u = 
\left(
\frac{\partial^2}{\partial x^2} + \mu(x) \frac{\partial }{\partial x} +
\frac{\partial }{\partial t} \right) u = 0
\]
has the same necessary and sufficient conditions (\ref{eq:Ricatti_equations}) of Lie symmetries' existence. However, the symmetries are different.
\end{remark}
\begin{remark}
	It is easy to see that for drift functions $\mu \in \Ffirst$ or $\mu \in \Fsecond$ FPKE has 6 independent groups of symmetries. It is known that second order parabolic equation with six group of symmetries can be reduced to the heat equation (see e.g. \cite{Goard}). This means that we can apply Theorem \ref{thm:main_ch_mes} with a standard Brownian Motion $W_t$. 
	
	At the same time, FPKE for the process $X_t$ with drift $\mu \in \Fthird$ or $\mu \in \Ffourth$ has 4 groups of symmetries. This means that FPKE can be reduced to the equation 
	\[
	\frac{1}{2} \frac{\partial^2 u}{\partial x^2} - \frac{Q}{x^2} u 
	- \frac{\partial u}{\partial t} =0.
	\]
	Therefore, we can apply Theorem \ref{thm:main_ch_mes} and set $Y_t$ is equal to the $\delta$-dimensional Bessel process $Z_t^\delta$:
	\begin{equation}
	\label{eq:Bessel_process_def}
		dZ_t^\delta = \frac{\delta - 1}{2 Z_t^\delta} dt + dW_t.	
	\end{equation}
\end{remark}
Any Ricatti equation from (\ref{eq:Ricatti_equations}) can be explicitly solved. We find solution $\mu(x)$ in the form 
\begin{equation}
\label{eq:mu_theta_link}
\mu(x) = \ln' \left(\THn(x) \right), \quad \THn(x) = e^{\int \mu(x) dx}, \quad 
\mu \in \Fn, \quad n = 1..4,
\end{equation}
function $\THn(x)$ makes linearization of Ricatti equation  
\begin{equation}
\label{eq:theta}
\THn'(x) - \left(\mu'(x) + \mu^2(x) \right) \THn(x) = 0.
\end{equation}
Therefore, we can represent diffusion processes $X_t$ in the following form 
\begin{equation}
\label{eq:SDE_mu}
dX_t = \frac{\THn'\left(X_t\right)}{\THn\left(X_t\right)} dt + dW_t, \quad X_0= x_0, \quad n = 1..4.
\end{equation}
In next subsections we consider in the detail each family, derive closed form solutions, obtain non-trivial Lie symmetries, construct two-parametric transforms and prove its uniqueness.
\subsection{Drift functions from $\Ffirst: \mu^2 + \mu' = 4 \B x + 2 \C$}
Consider the process $X_t$ of the form (\ref{eq:SDE_mu}) and set $n = 1$.
\begin{proposition}
Solution of equation (\ref{eq:theta}) for $n=1$ is the following function
\begin{equation}
\label{eq:mu_F1}
\THfirst(x) = \left\{ {\begin{array}{l}
	c_1 Ai\left(\sqrt[3]{4\B}\left[ x + \C/2\B \right] \right) +
	c_2 Bi\left(\sqrt[3]{4\B}\left[ x + \C/2\B \right] \right), \\
	c_1 \sinh\left(\sqrt{2C} x\right) + c_2 \cosh\left(\sqrt{2C} x\right), \quad \B = 0,
\\
c_1 x + c_2, \quad \B=\C = 0,
\end{array}} \right.
\end{equation} 
where $Ai(z)$ and $Bi(z)$ are Airy functions, for definition and properties see \cite{AbraStegun}.
\end{proposition}
\begin{proof}
	Apply formulas 1 and 2 with $\lambda = \C$, $\gamma = 2\B$ from \cite[~ Appendix 4]{BorodinSalminen}.
\end{proof}
\begin{remark}
The case $\B =\C= 0$ corresponds to the standard Brownian Motion ($c_1 =0$) and 3-dimensional Bessel process ($c_2 = 0$). 
\end{remark}
Now we are ready to obtain all groups of symmetries for FPKE. 
\begin{proposition}
\label{prop:symm_F1}
Let $X_t$ be a solution of (\ref{eq:SDE_mu}) for $n=1$. FPKE has the following Lie symmetries only 
\begin{eqnarray}
\label{eq:explicit_symm_F1}
% first
\left. { \begin{array}{l}
%%%%% first
\Sym^{\epsilon}_1: \quad 
\frac{u(x,t)}{\U \left(\X_{\Sym_{1}}, \T_{\Sym_{1}} \right)} =  
\frac{\THfirst \left( x\right)} 
     {\THfirst\left( \X_{\Sym_{1}} \right)}
\frac{
	e^{- \frac{(x + \B t^2)^2}{2 t}+\frac{2\B^2 t^3}{3} - \C t}}{
	e^{- \frac{(\X_{\Sym_{1}} + \B \T_{\Sym_{1}}^2)^2}{2 \T_{\Sym_{1}}}+\frac{2\B^2 \T_{\Sym_{1}}^3}{3} - \C \T_{\Sym_{1}}}}
\frac{t^{-1/2}}{\T_{\Sym_{1}}^{-1/2}}
\\ 
% second
\Sym_2^{\epsilon}: \quad 
\frac{u(x,t)}{\U \left(\X_{\Sym_{2}}, \T_{\Sym_{2}} \right)} = \frac{\THfirst \left( x\right)}{\THfirst \left( \X_{\Sym_{2}} \right)} 
\frac{e^{-2\B x t  + \frac{2 \B^2 t^3}{3}- \C t}}
{e^{-2\B \X_{\Sym_{2}} \T_{\Sym_{2}}  + \frac{2 \B^2 \T_{\Sym_{2}}^3}{3}- \C \T_{\Sym_{2}}}}
\\ 
% third
%%%%% first
\Sym_3^{\epsilon}: \quad 
\frac{u(x,t)}{\U(x - \epsilon t, t)} = \frac{\THfirst\left(x\right)}{\THfirst\left(x - \epsilon t\right)} \frac{e^{-\frac{(x + \B t^2)^2}{2t}}}{e^{-\frac{(x - \epsilon t + \B t^2)^2}{2t}}},
\\ 
%%%%% fourth
\Sym_4^{\epsilon}: \quad 
\frac{u(x,t)}{\U(x - \epsilon, t)} = 
\frac{\THfirst\left(x\right)}{\THfirst\left(x - \epsilon\right)} 
\frac{e^{-2\B x t}}{e^{-2\B \left(x - \epsilon\right) t}},
\\ 
%%%%% fifth
\Sym_5^{\epsilon}: \quad 
\frac{u(x,t)}{\U(x, t + \epsilon)} = 1, 
\\ 
%%%%% sixth
\Sym_6^{\epsilon}: \quad
\frac{u(x,t)}{\U(x, t)} = e^\epsilon
\\ 
%%%%% beta
\Sym_\beta: \quad 
u(x,t) =  \beta(x,t) + \U\left(x, t \right)
\end{array} }\right.
\end{eqnarray}
Here $\X_{\Sym_{1}}$, $\T_{\Sym_{1}}$, $\X_{\Sym_{2}}$, $\T_{\Sym_{2}}$ are defined by the following formulas
\begin{equation}
\left. { \begin{array}{c}
\X_{\Sym_{1}} = \frac{x - \B t^2 }{1+\epsilon t}  + \frac{\B t^2 }{(1+\epsilon t)^2}, \quad \T_{\Sym_{1}} = \frac{t}{1 + \epsilon t},
\\
\X_{\Sym_{2}} = \left(x - \B t^2 \right)e^{-\epsilon/2} + \B t^2  e^{-2\epsilon}, 
\quad  \T_{\Sym_{2}} = te^{-\epsilon}. 
\end{array} } \right.
\end{equation}
\end{proposition}
\begin{proof}
Proof of this result is a textbook, so we just summarize only main points. At first we solve the system (\ref{eq:main_system_for_symmetries}) and obtain components $\xi$, $\tau$, $\varphi$ of the symmetry generator $\textbf{v}$. Further we make exponentiation (\ref{eq:exponentiation}) and after tedious algebra get formulas (\ref{eq:explicit_symm_F1}). 
\end{proof}

\begin{proposition}
\label{prop:two_params_F1}
Any Lie symmetry (\ref{eq:Basic_transform}) of FPKE such that $\T(0) = 0$ is represented as one of the following two-parametric symmetries or its composition:
\begin{eqnarray}
\label{eq:two_parametric_symmetry_explicit_form_F1_s1}
\left. { \begin{array}{l}
\Sym_{1}^{\alpha, \beta}: \quad \frac{u(x,t)}{\U(\X, \T)} = \frac{\THfirst \left( x\right)} 
{\THfirst\left( \X \right)}
\frac{
	e^{- \frac{(x + \B t^2)^2}{2 t}+\frac{2\B^2 t^3}{3} - \C t}}{
	e^{- \frac{(\X + \B \T^2)^2}{2 \T}+\frac{2\B^2 \T^3}{3} - \C \T}}
\frac{t^{-1/2}}{\T^{-1/2}} ,
\\ 
\label{eq:two_parametric_symmetry_explicit_form_F1_s2}
\Sym_{2}^{\alpha, \beta}: \quad  \frac{u(x,t)}{\U(x - \alpha - \beta t, t)} = e^{-2\alpha \B t }
\frac{\THfirst\left(x\right)}{\THfirst\left(x - \alpha - \beta t\right)} 
\frac{e^{-\frac{(x + \B t^2)^2}{2t}}}{e^{-\frac{(x - \alpha - \beta t + \B t^2+ \alpha)^2}{2t}}},\quad \end{array}} \right.
\end{eqnarray}
Here $\X$ and $\T$ are defined as
\begin{equation}
\label{eq:X_def_F1}
\X = \alpha \frac{x - \B t^2 }{1+\alpha \beta t}  + \frac{\B \alpha^4 t^2 }{(1+\alpha \beta t)^2}, \quad \T = \frac{\alpha ^2 t}{1 + \alpha \beta t}
\end{equation}
\end{proposition}
\begin{proof}
We will only point out the main ideas for the proof. It is easy to check that any repeated application of the symmetries $\Sym_{1}^{\epsilon}$ or $\Sym_{2}^{\epsilon}$ from (\ref{eq:explicit_symm_F1}) after the symmetry $\Sym_{1, 2}^{\alpha, \beta}$ from (\ref{eq:two_parametric_symmetry_explicit_form_F1_s1}) holds structure of $\Sym_{1, 2}^{\alpha, \beta}$, i.e  $ \forall \alpha, \beta, \epsilon$, $\exists \hat{\alpha}, \hat{\beta}$: 
\[
\Sym_{1(2)}^{\epsilon} \circ \Sym_{1, 2}^{\alpha, \beta} = \Sym_{1, 2}^{\hat{\alpha}, \hat{\beta}}.  
\] 
Proof for the symmetry $\Sym_{3, 4}$ is fully analogical.
\end{proof}

\begin{thm}
\label{thm:F1_two_params}
Let $T^b$ be the first passage time to the boundary $b(t)$ by the process $X_t$ defined by (\ref{eq:SDE_mu}) with $n = 1$. Consider two-parametric family of curves defined over the curve $b(t)$:
\begin{equation}
g^{\alpha, \beta}(t) = \frac{1 + \alpha \beta t}{\alpha} b \left( \frac{\alpha^2 t}{1 + \alpha \beta t}\right)
+ \B t^2 - \frac{\alpha^3 \B t^2}{1 + \alpha \beta t}.
\end{equation}
Let $g(t)$ be an any curve from $g^{\alpha, \beta}$ and $T^g$ be the first hitting time (\ref{eq:FPT_g}). The boundary crossing identity between densities correspond to the stopping times $T^b$ and $T^g$ has the following representation:
\begin{equation}
\frac{\rho_{x_0}^{T^g}(t)}
{\rho_{\alpha x_0} ^{T^b} \left( \T \right)} 
= \frac{\alpha^2 e^{ \frac{\alpha \beta x_0^2}{2}}}{(1+\alpha \beta t)^{3/2}} \frac{\THfirst(\alpha x_0)}{\THfirst(x_0)} \frac{\THfirst\left(g(t) \right)}{\THfirst\left(b\left(\T \right) \right)}
\frac{
	e^{- \frac{(g(t) + \B t^2)^2}{2 t}+\frac{2\B^2 t^3}{3} - \C t}}{
	e^{- \frac{(b(\T) + \B \T^2)^2}{2 \T}+\frac{2\B^2 \T^3}{3} - \C \T}}.
\end{equation}
Here variable $\T$ is defined in (\ref{eq:X_def_F1}).  Symmetry $\Sym_{2}^{\alpha, \beta}$ from (\ref{eq:two_parametric_symmetry_explicit_form_F1_s1}) corresponds to the following two-parametric family of curves:
\begin{equation}
g^{\alpha, \beta}(t) = \alpha + \beta t + b(t).
\end{equation}
The law of $T^g$ to any boundary from $g \in g^{\alpha, \beta}(t)$ by the process $X_t$ can be explicitly represented from the following representation
\begin{equation}
\label{eq:Girsanov_type_mu_1}
	\frac{\rho_{x_0}^{T^g}(t)}
	{\rho_{x_0 - \alpha} ^{T^b} \left( t \right)} = 
	e^{-2\alpha \B t}\frac{\THfirst(x_0 - \alpha )}{\THfirst(x_0)} \frac{\THfirst\left(g(t)\right)}{\THfirst \left(b(t)\right)}
	\frac{e^{-\frac{\left( g(t) + \B t^2\right)^2}{2t}}}{e^{-\frac{\left( b(t) +\alpha + \B t^2\right)^2}{2t}}}
\end{equation}
\end{thm}
\begin{proof}
	Apply Theorem \ref{thm:main} with symmetries from (\ref{eq:two_parametric_symmetry_explicit_form_F1_s1}).
\end{proof}
\begin{thm}
	\label{thm:F1_to_BM}
	Let $X_t$ be a stochastic process of the form (\ref{eq:SDE_mu}) with $ n = 1$. Let $b(t)$ and $g(t)$ also be smooth functions of time such that 
	\[ 
	g(t) = b(t) + \B t^2.
	\] 
	Define stopping times $T^g$ and $T^b$ for the process $X_t$ and for a standard Brownian motion $W_t$ respectively:
	\begin{equation}
	T^g = \inf \left\{t \geq 0, \, X_t \leq g(t) \right\}, \quad  
	T^b = \inf \left\{t \geq 0, \, W_t \leq b(t) \right\}
	\end{equation}
	Then, there exists the following identity between the densities $\rho^{T^g}_{x_0} (t)$ and $\rho^{T^b}_{x_0} (t)$:
	\begin{eqnarray}
\label{eq:rho_F1_toBM}
\rho^{T^g}_{x_0} (t) = \frac{\THfirst\left(g(t)\right)}{\THfirst(x_0)} e^{\frac{2\B t^3}{3} - \C t - 2\B t g(t) } \rho^{T^b}_{x_0} (t)
\end{eqnarray}
	Laplace transform of the stopping moment $T^g$ can be represented by:
	\begin{equation}
	\mathbb{E}_{x_0}\left[e^{-\lambda T^g}\right] = \frac{\mathbb{E}_{x_0}\left[ \THfirst\left(W_{T^b} + \B (T^b)^2 \right) e^{-2\B T^b W_{T^b}  -\frac{4\B (T^b)^3}{3}-(\lambda + \C) T^b}\right]}{\THfirst(x_0)}
	\end{equation}
\end{thm}
\begin{proof}
Let $\LLinf$  and $\GGinf$ be infinitesimal generators of diffusion process $X_t$ and standard Brownian motion $W_t$ respectively. We can check by the direct calculations that the solutions of equations $\left( \LL - \partial_t\right) u = 0$ and  $\left( \GG - \partial_\T\right) \U = 0$ has the following relation of the form (\ref{eq:Basic_transform}):
\[
u(x,t) = \THfirst(x) e^{\frac{2\B^2}{3} t^3 - Ct - 2 \B x t} \U(x-\B t^2, t).	 
\]
Application of Theorem \ref{thm:main_ch_mes} finishes the proof.
\end{proof}

\begin{example}
	Let $W_t^{({\omega})}$ be a Brownian motion with drift $\omega$ and $M_t^{(\omega)} = \sup\left\{W_s : s \leq t \right\}.$ The process $X_t = 2 M_t^{(\omega)} - W_t^{(\omega)}$ (see \cite{Rogers Pitman}) solves SDE
\begin{equation}
	dX_t = |\omega| \coth\left(|\omega| X_t \right) dt + dW_t, \quad X_0 =x_0.  
\end{equation} 
Drift $|\omega| \coth\left(|\omega|x\right) \in \Ffirst$ with parameters $\B = 0$, $\C = \omega^2/2$, hence $\theta(x)  = \sinh(|\omega|x)$. Apply Theorem \ref{thm:F1_to_BM} for the stopping times: 
\[
T^g = \inf \left\{t \geq 0, \, X_t \leq b \right\},
\quad T^b = \inf \left\{t \geq 0, \, W_t \leq b \right\}.
\]
In the result we have
\begin{eqnarray}
	\rho^{T^g}_{x_0} (t) 
	&=& \frac{\sinh\left(|\omega|b\right)}{\sinh\left(|\omega|x_0\right)} e^{- \omega^2 t / 2} \rho^{T^b}_{x_0} (t) 
	\\ \nonumber
	&=&  \frac{|b - x_0|}{\sqrt{2\pi t^3}} \frac{\sinh\left(|\omega|b\right)}{\sinh\left(|\omega|x_0\right)} e^{- \omega^2 t / 2 - \frac{(b-x_0)^2} {2t}}. 
\end{eqnarray}
Application of Theorem \ref{thm:F1_two_params} yields the explicit formula for the density of the stopping moment $ T^h = \inf \left\{t \geq 0, \, X_t \leq at+ b \right\}$:
\begin{equation}
\rho_{x_0}^{T^h} (t) = \frac{|b - x_0|}{\sqrt{2\pi t^3}} \frac{\sinh\left(|\omega| (at + b) \right)}{\sinh\left(|\omega|x_0\right)} 
e^{-\frac{\omega^2 + a^2}{2}t - a(b-x_0) - \frac{(b-x_0)^2} {2t}}
\end{equation}

%Then, we can easily derive probability that $X_t$ never hits level $b < x_0$:
%\begin{equation}
%\mathbb{P}_{x_0} (X_t > b, \, \forall t> 0) = \frac{e^{2|\omega| x_0} - e^{2|\omega| b}}{e^{2|\omega| x_0} - 1}.
%\end{equation}
\end{example}

\begin{example}
Consider Pearson diffusion process $U_t$ defined by SDE:
\begin{equation}
\label{eq:Pearson_diffusion}
dU_t = \left(\alpha U_t +\beta \right)dt + \sqrt{r U_t^2 + pU_t + q} dW_t,\quad U_0 = u_0, \quad  r > 0. 
\end{equation}
We are looking for closed form formula for the distribution of the first passage time $ T^h = \inf \left\{ t \geq 0, \, U_t  \leq h \right\} $.
Denote by $d = -(p^2 - 4 r q)$ and suppose that $d > 0$. Application of Lamberti transform turns out to the process $X_t$:
\begin{eqnarray}
\nonumber
X_t &=& \int\frac{du}{\sqrt{r u^2 +p u + q}} |_{u = U_t} = \frac{1}{\sqrt{r}} \ln\left(\sqrt{rU_t^2 + pU_t + q} + \sqrt{r} U_t +\frac{p}{2 \sqrt{r}}\right)
\\ \nonumber
U_t &=&-\frac{p}{2\sqrt{r}} + \frac{e^{\sqrt{r}X_t} - d e^{-\sqrt{r}X_t}  }{2\sqrt{r}}, \quad \sqrt{rU_t^2 + p U_t  + q} = \frac{e^{\sqrt{r}X_t} + d e^{-\sqrt{r}X_t}  }{2}.
\end{eqnarray}
In the general case FPKE does not have Lie symmetries. It is easy to check that symmetries exist if and only if $\alpha$ and $\beta$ has the following form
\[
\alpha = \frac{3}{4} r, \quad \beta = \frac{3}{4} p,
\] 
The process $X_t$ solves the following SDE ($\theta(x) = {e^{\sqrt{a}x} - d e^{-\sqrt{a}x}}$):
\begin{equation}
dX_t = \sqrt{r}\frac{e^{\sqrt{r}X_t} + d e^{-\sqrt{r}X_t}}{e^{\sqrt{r}X_t} - d e^{-\sqrt{r}X_t}} dt + dW_t.
\end{equation}
The stopping time $T^h$ has the same distribution that the stopping time $ T^g = \inf \left\{ t \geq 0, \, X_t \leq b \right\}$. Application of Theorem \ref{thm:F1_to_BM} yields the formula
\begin{equation}
\nonumber
\rho_{x_0}^{T^g} (t) = 
\frac{|b- x_0|}{\sqrt{2\pi t^3}} 
\frac
{e^{\sqrt{r} b} - d e^{-\sqrt{r}b}} 
{e^{\sqrt{r} x_0} - d e^{-\sqrt{r} x_0}}
e^{-\frac{r t}{2}  - \frac{(b-x_0)^2}{2t}}.
\end{equation}
After the following substitutions
\begin{equation}
\begin{array}{c} 
2\sqrt{r} u_0 + q = e^{\sqrt{r} x_0} - d e^{-\sqrt{r} x_0}, \quad 
2\sqrt{r} h + q = e^{\sqrt{r} b} - d e^{-\sqrt{r} b}
\\ 
b - x_0 = 
\frac{1}{\sqrt{r}}
\ln \left(
\frac{2\sqrt{r} h + q + \sqrt{\left(2\sqrt{r} h + q\right)^2 + 4d}}
{2\sqrt{r} u_0 + q + \sqrt{\left(2\sqrt{r} u_0 + q\right)^2 + 4d}}
\right) 
\end{array}
\end{equation}
we have
\begin{equation}
\begin{array} {c}
\rho_{u_0}^{T^h} (t) = 
\frac{1}{\sqrt{2\pi r t^3}} 
\ln \left(
\frac{2\sqrt{r} h + q + \sqrt{\left(2\sqrt{r} h + q\right)^2 + 4d}}
{2\sqrt{r} u_0 + q + \sqrt{\left(2\sqrt{r} u_0 + q\right)^2 + 4d}}
\right) 
\frac {2\sqrt{r} h +q} {{2\sqrt{r} u_0 + q}} 
\times
\\
\times
\exp \left\{ -\frac{r t}{2}
- \frac{1}{2 r t} \ln^2 \left(
\frac{2\sqrt{r} h + q + \sqrt{\left(2\sqrt{r} h + q\right)^2 + 4d}}
{2\sqrt{r} u_0 + q + \sqrt{\left(2\sqrt{r} u_0 + q\right)^2 + 4d}}
\right)  
\right\}
\end{array}
\end{equation}
\end{example}

\begin{remark}
Consider the first passage time $T^b$ of quadratic boundary $b(t) = a- \B t^2$ by a Brownian Motion. The distribution of $T^b$ can be explicitly represented in terms of stopping time 
$ T^g = \inf \left\{t \geq 0 ,\, X_t \leq a \right\} $, where process $X_t$ is a solution to the following SDE:
\[
dX_t = \sqrt[3]{4\B}\frac{Ai'\left(\sqrt[3]{4 \B} X_t \right)}{Ai\left(\sqrt[3]{4 \B} X_t \right)} dt +
dW_t.
\]
At the same time, law of $T^g$ can be explicitly derived by Laplace transform and spectral expansions method, see \cite{Linetsky}. The probabilistic approach to analysis of stopping time $T^b$ is presented in \cite{Groeninboom} and \cite{Salminen}. 
\end{remark}
\subsection{Drift functions from $\Fsecond: \mu^2 + \mu' = \A x^2 + 4 \B x + 2 \C$}
Now we turn to the process $X_t$ of the form (\ref{eq:SDE_mu}) with $n = 2$. 
\begin{proposition}
	The solution of equation (\ref{eq:theta}) for $n= 2$ is the following function
\begin{equation}
\THsecond(x) = c_1 D_{2\nu} \left(\sqrt[4]{4\A} \left[ x + \frac{2 \B}{\A} \right] \right) 
+ c_2 D_{2\nu} \left(-\sqrt[4]{4\A} \left[ x + \frac{2 \B}{\A} \right] \right), 
\end{equation}
where $D_{\nu}(\pm z)$ are parabolic cylinder functions, for definitions and properties see \cite{AbraStegun}, constant $\nu$ is equal
\begin{equation}
\nu =\frac{1}{2\sqrt{\A}} \left( \frac{2\B^2}{\A} - \frac{\sqrt{\A}}{2} -  \C   \right),
\end{equation}  
\end{proposition}
\begin{proof}
	Apply formula 3 with $\gamma^2 = \A$, $\lambda = \C - 2\B^2 / \A.$ from \cite[~ Appendix 4]{BorodinSalminen} . 
\end{proof}

\begin{proposition}
Let $X_t$ be a diffusion process of the form (\ref{eq:SDE_mu}) with $ n = 2$. Corresponded FPKE has the following Lie symmetries only
\begin{eqnarray}
\label{eq:explicit_symm_F2}
\left. { \begin{array}{l}
\Sym_{1, 2}^{\epsilon}: \,
\frac{u(x,t)}
     {\U \left( \X, \T \right)} = 
\frac{\THsecond\left(x\right)}{\THsecond\left(\X\right)} 
\frac{e^{2\sqrt{\A} \left(\nu - \frac{1}{4} \pm \frac{1}{4} \right) t 
		\mp \frac{\sqrt{\A}}{2} \left(x + 2\B / \A \right)^2}}
{e^{2\sqrt{\A} \left(\nu - \frac{1}{4} \pm \frac{1}{4} \right) \T 
		\mp \frac{\sqrt{\A}}{2} \left(\X + 2\B / \A \right)^2}},
\\ 
\Sym_{3, 4}^{\epsilon}: \, 
\frac{u(x,t)}
     {\U\left(x - \epsilon e^{\pm \sqrt{\A}t}, t \right)} =  
\frac{\THsecond\left(x\right)}
     {\THsecond\left(x - \epsilon e^{\pm \sqrt{\A}t} \right)} 
\frac{e^{ \mp \frac{\sqrt{A}}{2}\epsilon e^{\pm \sqrt{\A}t} \left[ x + \frac{2\B}{\A} \right]}}
	 {e^{ \pm \frac{\sqrt{A}}{2}\epsilon e^{\pm \sqrt{\A}t} \left[ x - \epsilon e^{\pm \sqrt{\A}t} + \frac{2\B}{\A} \right]}},
\\ 
\Sym_{5}^{\epsilon}:, \, \frac{u(x,t)}{\U(x, t + \epsilon)} =  1,
\\ 
\Sym_{6}^{\epsilon}:, \,  \frac{u(x,t)}{\U(x, t)} =  e^\epsilon, 
\\ 
\Sym_{\beta}, \, u(x,t) =  \beta(x,t) + \U\left(x, t \right). 
\end{array} } \right.
\end{eqnarray}
Here 
\[
\X_{\Sym_{1, 2}} = \frac{x + 2\B / \A}{\sqrt{1+\epsilon e^{\pm 2\sqrt{\A t}}}} - \frac{2\B}{\A}, \quad 
\T_{\Sym_{1,2}}= -\frac{\ln\left( \epsilon + e^{\mp 2\sqrt{\A}t}\right) }{2\sqrt{\A}}.
\]
\end{proposition}
\begin{proof}
Proof is fully analogical to the proof of Proposition \ref{prop:symm_F1}.
\end{proof}
\begin{proposition}
\label{prop:two_params_F2}
Any Lie symmetry (\ref{eq:Basic_transform}) such that $\T(0) = 0$ is one of the following two-parametric symmetries or its composition:
\begin{eqnarray}
\label{eq:two_parametric_symmetry_explicit_form_F2_s1}
\left. { \begin{array}{l}
\Sym_{1}^{\alpha, \beta}: \, \frac{u(x,t)}{\U \left(
	\X_{\Sym_{1}},
	\T_{\Sym_1} \right)} 
%&=& \frac{\THsecond\left( x \right)}{\THsecond\left(\X^{\alpha, \beta}\right)} \frac{\T_{+}^{\nu}}{ \T_{-}^{\nu - 1/2}}
%e^{ 
%	-\frac{\sqrt{\A}}{2}\frac{\left( \X^{\alpha, \beta} + 2\B / \A \right)^2 \left[(\beta + 1)\T_{+} - (\alpha + 1)\T_{-} \right] }{1 +\alpha + \beta}} 
%\frac{u(x,t)}{\U \left(\X^{\alpha, \beta}, \T^{\alpha, \beta} \right)} 
= \frac{\THsecond\left( x \right)}{\THsecond\left(\X_{\Sym_{1}}\right)} 

\frac{e^{ 
		-\frac{\sqrt{\A} \left(x + 2\B / \A \right)^2 \left[(\beta + 1) \T_{+} - (\alpha + 1)\T_{-} \right] }{2\T_{+} \T_{-}}}}
	{\T_{+}^{-\nu} \T_{-}^{\nu - 1/2}}  
\\
\label{eq:two_parametric_symmetry_explicit_form_F2_s2}
\Sym_{2}^{\alpha, \beta}: \, 
\frac{u(x,t)}{\U(\X,t)} =  \frac{\THsecond\left( x \right)}
{\THsecond\left(\X\right)}
\frac{e^{ -\frac{\sqrt{\A}}{2} \left( \alpha e^{\sqrt{\A} t} - \beta e^{-\sqrt{\A} t}\right) \left[
		x +2\B / \A \right] }}
{e^{ +\frac{\sqrt{\A}}{2} \left( \alpha e^{\sqrt{\A} t} - \beta e^{-\sqrt{\A} t}\right) \left[
		\X + 2\B / \A \right] }}
\end{array} } \right.
\end{eqnarray}
Here $\X_{\Sym_{1}}$, $\X_{\Sym_{2}}$, $\T_{-}$, $\T_{+}$ and $\T_{\Sym_{1}}$ are defined as
\begin{equation}
\label{eq:T_plus_minus}
\begin{array}{c}
\X_{\Sym_{1}} = \frac{(x + 2\B/\A)\sqrt{\alpha + \beta + 1}}{\sqrt{ \T_{+} \T_{-}}} - \frac{2 \B}{\A}, \quad 
\X_{\Sym_2} = x - \alpha e^{\sqrt{\A} t} - \beta e^{-\sqrt{\A}t},
\\
\T_{-} = \beta + 1 + \alpha e^{-\ExpArgScnd}, \quad \T_{+} = \alpha + 1 + \beta e^{\ExpArgScnd},
\\
 \T_{\Sym_{1}} = t + \frac{\ln \left(\T_{-} / T_{+}\right)}{2\sqrt{A}} .
\end{array}
\end{equation}
%\begin{equation}
%h^{\alpha, \beta}_{\Fsecond}(x,t) = 
%\frac{\THsecond\left( x \right) e^{ -\sqrt{\A} \left[
%		\alpha \beta + \frac{1}{2}\left(\beta^2 e^{-\ExpArgScnd} - \alpha^2 e^{\ExpArgScnd} \right)
%		+ \left(x + \frac{2\B}{\A}\right)  \left( \alpha e^{\sqrt{\A} t} - \beta e^{-\sqrt{\A} t}\right) 
%		\right] }	
%}
%{\THsecond\left(x - \alpha e^{\sqrt{A}t} - \beta e^{-\sqrt{\A} t}\right)}
%\end{equation}
\end{proposition}
\begin{proof}
Proof is fully analogical to the proof of Proposition \ref{prop:two_params_F1}.
\end{proof}
\begin{thm}
	\label{thm:F2_two_params}
Let $T^b$ be the first passage time to the boundary $b(t)$ by the process $X_t$ defined by (\ref{eq:SDE_mu}) with $n=2$.  Consider the following two-parametric family of the curves defined over the curve $b(t)$:
\[
g^{\alpha, \beta} (t) = \frac{\sqrt{\T_{+} \T_{-}}}{\sqrt{1 + \alpha + \beta}}
\left[b \left(t + \frac{\ln\left(\T_{-} / T_{+}\right)}{2\sqrt{\A}}\right) + \frac{2\B}{\A} \right]
-\frac{2\B}{\A}.
\] 
Then, the distribution of the first hitting time to the any curve $g(t)$ from $g^{\alpha, \beta}(t)$ can be explicitly represented in terms of the distribution of $T^b$:
\begin{equation}
\left. { \begin{array} {c} 
\frac{\rho_{x_0}^{T^g}(t)}
{\rho_{\frac{x_0 + 2\B/\A}{\sqrt{\alpha + \beta + 1}}}^{T^b}
	\left(t + \frac{\ln\left(\T_{-} / T_{+}\right)}{2\sqrt{\A}}\right)
} = 
\frac{\T_{-}^{-\nu} \T_{+}^{\nu -1/2}}{\sqrt{\alpha + \beta + 1}} \frac{\THsecond\left(g(t)\right)}{\THsecond\left(b\left(t + \frac{\ln\left(\T_{-} / T_{+}\right)}{2\sqrt{\A}}\right)\right)} \times
\\
\times \frac{\THsecond \left( \frac{x_0 + 2\B/\A}{\sqrt{\alpha + \beta + 1}}\right)}
{\THsecond\left(x_0\right)}
\frac{e^{ 
		-\frac{\sqrt{\A} \left(g(t) + 2\B / \A \right)^2 \left[(\beta + 1) \T_{+} - (\alpha + 1)\T_{-} \right] }{2\T_{+} \T_{-}}}}
{e^{ 
		-\frac{\sqrt{\A} \left(x_0 + 2\B / \A \right)^2 \left[\beta- \alpha \right] }{2\left(1 + \alpha + \beta\right)}}}
\end{array} }
\right.
\end{equation}

Symmetry $\Sym_{2}^{\alpha, \beta}$ from (\ref{eq:two_parametric_symmetry_explicit_form_F1_s2}) corresponds to the following two-parametric family of curves:
\begin{equation}
g^{\alpha,\beta} = b(t) + \alpha e^{\sqrt{\A}t} + \beta e^{-\sqrt{\A}t}.
\end{equation}
Boundary crossing identity has the form 
\begin{equation}
\frac{\rho_{x_0}^{T^g}(t)}
{\rho_{x_0 -\alpha - \beta}^{T^b} \left(t\right) } = 
\frac{\THsecond\left( g(t) \right)}
{\THsecond\left(b(t) \right)}
\frac{e^{ -\frac{\sqrt{\A}}{2} \left( \alpha e^{\sqrt{\A} t} - \beta e^{-\sqrt{\A} t}\right) \left[
		g(t) +2\B / \A \right] }}
{e^{ +\frac{\sqrt{\A}}{2} \left( \alpha e^{\sqrt{\A} t} - \beta e^{-\sqrt{\A} t}\right) \left[
		b(t) + 2\B / \A \right] }}
\end{equation}
\end{thm}
\begin{proof}
	Proof is fully analogical to the proof of Theorem \ref{thm:F1_two_params}.
\end{proof}

\begin{thm}
\label{thm:F2_to_BM}
Let $X_t$ be a process (\ref{eq:SDE_mu}) with $n =2$.  Let's consider the following stopping times
\[
  T^b = \inf \left\{t \geq 0, \, W_t \leq b(t)\right\}, 
  \quad 
  T^g = \inf \left\{t \geq 0, \, X_t \leq g(t)\right\}
\]
where functions $b(t)$ and $g(t)$ have the following relation:
	\[
	g(t) = \frac{e^{-\sqrt{\A}t}}{\sqrt[4]{4\A}}b\left(e^{2\sqrt{\A} t} - 1\right) - \frac{2\B}{\A}.
	\]
The density $\rho^{T^g}$ of the distribution of $T^g$ can be explicitly represented from the following identity:
	\begin{equation}	
	\label{eq:rho_F2_toBM}
	\frac{\rho^{T^g}_{x_0} (t)}
	{\rho^{T^b}_{w_0} \left(e^{2\sqrt{\A} t}  - 1\right)} = 
	\frac{2\sqrt{\A}}{e^{-2(\nu +1) \sqrt{\A} t}}
	\frac{\THsecond\left(g(t)\right)}{\THsecond(x_0)}
	\frac{e^{\frac{\sqrt{\A}\left(g(t) + 2\B / \A\right)^2}{2}}}{e^{\frac{\sqrt{\A}\left(x_0 + 2\B / \A\right)^2}{2} }}. 
	\end{equation}
Laplace transform of $T^g$ has the following representation:
\begin{equation}
\begin{array}{r} 
\mathbb{E}_{x_0}\left[e^{-\lambda T^g} \right] =
\frac{e^{-w_0^2 / 4}}{ \THsecond(x_0) }\mathbb{E}_{w_0}\left[
\frac{e^{ W_{T^b}^2 (T^b + 1)  / 4}}{\left(T^b + 1 \right)^{-\nu +\lambda / 2\sqrt{\A}} }
\THsecond\left(\frac{\sqrt{T^b + 1}W_{T^b}}{\sqrt[4]{4\A}} - \frac{2\B}{\A} \right)\right],
\\
w_0 = \sqrt[4]{4\A} (x_0 + 2\B / \A).
\end{array}
\end{equation}
\end{thm}
\begin{proof}
Proof is fully analogical to the proof of Theorem \ref{thm:F1_to_BM}. Basic relation (\ref{eq:Basic_transform}) has the following form: 
\begin{equation}
\nonumber
u(x,t) = \THsecond(x) e^{\frac{\sqrt{\A}\left(x + 2\B / \A\right)^2}{2} + \sqrt{\A} (2 \nu +1) t} 
\U \left(
\sqrt[4]{4\A} e^{\sqrt{\A}t} \left(x + \frac{2\B}{\A}\right), e^{2\sqrt{\A} t}  - 1
\right).
\end{equation}
\end{proof}
\begin{example}
	Consider Ornstein--Uhlenbeck process with zero mean and unit variance:
	\[
	dX_t= -\lambda X_t dt + dW_t, \quad X_0 = x_0.
	\]
	For this process we have $\A = \lambda ^2$, $\B = 0$, $\C = -\lambda/2$ and $\theta(x) = e^{-\frac{ \lambda x^2}{2}}$. Consider the following stopping times:
	\begin{equation}
	\nonumber
	\begin{array}{c}
		T^b = inf\left\{t \geq 0, \,\, W_t \geq b(t)  \right\}, \, W_0 = -\sqrt{2\lambda} x_0, 
		\,b(t)= a \sqrt{2\lambda(t+1)}, 
		\\
	T^g = \inf\left\{t\geq0, \, X_t \geq g(t)\right\}, \quad g(t) \equiv a,
	\end{array}
	\end{equation}
	The law of $T^b$ is known and can be found in \cite{Novikov Korzakhia} ($a = -\sqrt{2\lambda}x_0$, $b = a\sqrt{2\lambda}$, $c = 1$).
	\begin{equation}
	\label{eq:rho_sq_root_Wt}
	\rho^{T^b}_{-\sqrt{2\lambda} x_0}(t) = -\sum_{n = 1}^{\infty}\frac{HK(v_n, x \sqrt{2\lambda}) }{\partial/\partial_v HK(v_n, a\sqrt{2\lambda})} \left(t + 1\right)^{-v_n - 1}.
	\end{equation}
	Here function $HK(v, z)$ is defined as
	\begin{eqnarray}
	\label{eq:HKtoD}
	HK(v, z) &=& \int_{0}^{\infty} exp(zt - t^2/2) t^{-2v-1} dt,
	\\ \nonumber
	&=& D_{2\nu}(-z) exp(z^2 /4) / \Gamma(-2v), \quad \Re(v) <0.
	\end{eqnarray}
	Application of Theorem \ref{thm:F2_to_BM} yields formula for the law of stopping time $T^g$: 
	\[
		\rho_{x}^{T^g} (t) = -2\lambda \sum_{n = 1}^{\infty}\frac{HK(v_n, x\sqrt{2\lambda}) }{\partial/\partial_v HK(v_n, a\sqrt{2\lambda})} \exp\left(-2 v_n \lambda t \right).
	 \]
     Denote $v_{n, -a\sqrt{2\lambda}} = 2 v_n$ and apply relation (\ref{eq:HKtoD}) between function $HK(v,z)$ and parabolic cylinder function $D_{\nu}(z)$ (in the result we have formula from \cite{OU_density}):  	
	\[
		\rho_{x_0}^{T^g}(t) = - \lambda e^{\lambda(x_0^2 - a^2)/2} \sum_{j = 1}^{\infty} \Theta_{j} (x_0),
	\]
	Here
	\[
	\Theta_{j} \left(z\right) = \frac{ D_{v_{j, -a\sqrt{2\lambda}}}\left( -z \sqrt{2\lambda} \right) }
	{D'_{v_{j, -a\sqrt{2\lambda}}}\left( -a\sqrt{2\lambda}\right)} \exp(-\lambda v_{j, -a\sqrt{2\lambda}} t)
	\]
	and  $D'_{v_{j,b}}(b) = \frac{\partial D_v(b)}{\partial v}|_v =v_{j,b}$.  
	
	Now we turn to the following stopping times 
\begin{equation}
\nonumber
\begin{array}{c}
T^h = \inf\left\{ t\leq 0, \, X_t \leq a + \alpha e^{-\lambda t} + \beta e^{\lambda t} \right\}, 
\\ 
T^q = \inf\left\{ t\leq 0, \, X_t \leq \sqrt{\frac{\left( \alpha + 1 + \beta e^{2\lambda t}\right) 
		\left( \alpha e^{-2\lambda t} + 1 + \beta \right)}
	{1 + \alpha + \beta}}\right\}.
\end{array}
\end{equation}
The distribution of $T^h$ can be found in \cite{Bounocore}. Application of Theorem \ref{thm:F2_two_params} yields the following explicit formulas:
\begin{eqnarray}
\left. \begin{array}{c}
\rho_{x_0}^{T^h}(t) = 
- \lambda  
e^{-\lambda \alpha^2 \left( e^{2\lambda t} + 1 \right) 
	-2\lambda \alpha \left(e^{\lambda t} a + \beta - x_0\right)} 
\times
\\ 
\times
e^{\lambda((x_0 - \alpha - \beta)^2 - a^2)/2} 
\sum_{j = 1}^{\infty} \Theta_j(x_0 - \alpha - \beta)
\end{array} \right.
\end{eqnarray}
and
\begin{eqnarray}
\left. \begin{array}{c}
\rho_{x_0}^{T^q}(t) = -\lambda e^{\lambda \frac{x_0 (\beta - \alpha + 1) - a^2\left[2\beta ( \alpha  e^{-2\lambda t } + \beta + 1) + \alpha + \beta + 1\right] }{\alpha +\beta + 1}}
\\ 
\cdot
\sum_{j = 1}^{\infty} \Theta_{j} \left(\frac{x_0 \sqrt{2\lambda}}
{\alpha + \beta + 1} \right) 
\left(\frac{\alpha + 1 + \beta e^{-2\lambda t}}{\beta + 1 + \alpha e^{2\lambda t}} \right) ^{v_{j, -a\sqrt{2\lambda}} / 2}
\end{array} \right.
\end{eqnarray}
\end{example}
\begin{remark}
Based on Theorem \ref{thm:F2_to_BM} and first passage time for quadratic curve by Brownian motion we can derive closed form formulas for the distributions of first hitting times of the following curves 
\[
\alpha e^{\lambda t} + \beta e^{-\lambda t} + a e^{\pm 3 \lambda t}
\]
by Ornstein--Uhlenbeck process.
\end{remark}

\begin{example}
Now we turn to the following diffusion process 
\[
dX_t = \left( -3\lambda X_t  + \frac{2\sqrt{2 \lambda }}{\pi}  \frac{e^{-\lambda X_t^2}}{1 - erf \left(\sqrt{\lambda} X_t \right)}\right) dt + dW_t.
\]
This case corresponds to specifications $ \A = -\lambda^2$ $\B = 0$, $\C = 3 \lambda /2$ ($\nu = -1/2$). Function $\theta (x)$ is equal
\[
\theta (x) = D_{-1} (x) = e^{x^2 / 4} \sqrt{\frac{\pi}{2}}\left[ 1 - erf(x / \sqrt{2})\right],
\]
Here $erf(z)$ is the error function:
\[
erf(z) = \frac{2}{\sqrt{\pi}}\int_{0}^{z} e^{-t^2} dt.
\]
The distribution of stopping time $T^g = \inf\left\{t \geq 0, \, X_t \leq b \right\}$ can be explicitly found by application of Theorem \ref{thm:F2_to_BM} to the formula (\ref{eq:rho_sq_root_Wt})
\begin{eqnarray}
\rho_{x_0}^{T^g} (t) = - \lambda e^{\lambda(a^2 - x_0^2)/2} \frac{1- erf(\sqrt{\lambda} a)}{1- erf(\sqrt{\lambda} x_0)} 
\sum_{j = 1}^{\infty} \Theta_{j}(x_0).
\end{eqnarray}
\end{example}
\subsection{Drift functions from $\Fthird : \mu^2 + \mu' = 2 \C + \D/ x^2$}
Consider the process $X_t$ defined by SDE (\ref{eq:SDE_mu}) with $n = 3$. 
\begin{proposition}
The solution of equation (\ref{eq:theta}) for $n=3$ is the following function
\begin{equation}
\THthird(x) = \left\{ {\begin{array}{l}
	c_1 \sqrt{x} I_{\sqrt{1/4 + \D}} \left(\sqrt{2\C}x\right) + 
	c_2 \sqrt{x} K_{\sqrt{1/4 + \D}} \left(\sqrt{2\C}x\right) \\
	c_1 x^{1/2+\sqrt{1/4+\D}} + c_2 x^{1/2-\sqrt{1/4+\D}}, \, \C = 0
	\end{array}} \right.
\end{equation}
where $I_{\nu} (z)$ and $K_{\nu} (z)$ are modified Bessel functions, for definitions and properties see \cite{AbraStegun}.
\end{proposition}
\begin{proof}
Apply formulas 4 with $\gamma^2 = 1/4 + \D$ and $\lambda = \C$ from \cite[~ Appendix 4]{BorodinSalminen}.
\end{proof}
\begin{proposition}
\label{prop:symm_F3}
Let $X_t$ be a diffusion process of the form (\ref{eq:SDE_mu}) with $n = 3$.  FPKE has the following Lie symmetries only
\begin{eqnarray}
\label{eq:explicit_symm_F3}
\left. { \begin{array}{l}

\Sym_1^{\epsilon}: \quad 
\frac{u(x,t)}{\U \left(\frac{x}{1+\epsilon t}, \frac{t}{1 + \epsilon t}\right)} 
=  
\frac{\THthird \left( x\right)} 
{\THthird\left( \frac{x}{1+\epsilon t}\right)}
\frac{
	e^{
	\left[ \frac{x^2}{2t^2}  + \C \right] \left[ \frac{t}{1+\epsilon t} - t\right] }
 }{\sqrt{1+\epsilon t}} 
, 
\\ 
\Sym_2^{\epsilon}: \quad 
\frac{u(x,t)}{\U \left(xe^{-\epsilon/2},\, te^{-\epsilon}\right)} = \frac{\THthird \left( x\right)} 
{\THthird \left(xe^{-\epsilon/2} \right)}
e^{\C t\left[e^{-\epsilon} - 1\right]} 
, 
\\ 
\Sym_3^{\epsilon}: \quad 
\frac{u(x,t)}{\U(x, t + \epsilon)} = 1, 
\\ 
\Sym_4^{\epsilon}: \quad 
\frac{u(x,t)}{\U(x, t)} = e^\epsilon, 
\\ 
\Sym_\beta: \quad 
u(x,t) =  \beta(x,t) + \U\left(x, t \right).
\end{array} } \right.
\end{eqnarray}
\end{proposition}
\begin{proof}
Proof is fully analogical to the proof of Proposition \ref{prop:symm_F1}.
\end{proof}
\begin{proposition}
\label{prop:two_params_symmetries_F3}	
Any Lie symmetry (\ref{eq:Basic_transform}) such that $\T(0) = 0$ has the following two-parametric representation
	\begin{equation}
	\label{eq:two_parametric_symmetry_explicit_form_F3}
\Sym_1^{\alpha, \beta}: \quad \frac{u(x,t)}{\U\left(\frac{\alpha x}{1 + \alpha \beta t}, \frac{\alpha^2 t}{1+ \alpha \beta t}\right)} = 
	\frac{\THthird \left( x\right)} 
{\THthird\left( \frac{\alpha x}{1 + \alpha \beta t} \right)}
	\frac{	e^{-\frac{\alpha \beta x^2}{2 (1 + \alpha \beta t)} +
			\C t \left[\frac{\alpha^2}{1+\alpha \beta t}  - 1\right] }}{\sqrt{1+\alpha \beta t}}  
	\end{equation}
\end{proposition}
\begin{proof}
Proof is fully analogical to the proof of Proposition \ref{prop:two_params_F1}.
\end{proof}

\begin{thm}
	\label{thm:two_params_F3}
	Let $T^b$ be the first passage time to the boundary $b(t)$ by the process (\ref{eq:SDE_mu}) with $n=3$. Symmetry $\Sym_1^{\alpha, \beta}$ generates two parametric family of the curves 
	\begin{eqnarray}
	g^{\alpha, \beta}(t) = \frac{1 + \alpha \beta t}{\alpha} b\left( \frac{\alpha ^2 t}{1 + \alpha \beta t}\right) .
	\end{eqnarray}
	The density function of the first hitting time to any boundary $g(t)$ from $g^{\alpha, \beta}$ is represented as
	\begin{equation}
	\frac{\rho_{x_0}^{T^g} (t)}
	{\rho_{\alpha x_0} ^{T^b} \left( \frac{\alpha^2 t}{1+ \alpha\beta t}\right)}
	= 
	\frac{\alpha^2 	e ^ {\C t \left[\frac{\alpha^2}{1+\alpha \beta t}  - 1\right]}}{(1+\alpha \beta t)^{3/2}} \frac{\THthird(\alpha x_0)}{\THthird(x_0)} \frac{\THthird\left(g(t) \right)}{\THthird\left(b\left(\frac{\alpha^2 t}{1+\alpha \beta t}\right) \right)}
	\frac
	{e^{- \frac{\alpha \beta}{2} \frac{g^2(t)} {1 + \alpha \beta t}}}
	{e^{ -\frac{\alpha \beta}{2} x_0^2}}.
	\end{equation}
\end{thm}
\begin{proof}
Proof is fully analogical to the proof of Theorem \ref{thm:F1_two_params}.
\end{proof}

\begin{thm}
	\label{thm:F3_to_BES}
	Let $X_t$ be a stochastic process with drift $\mu(x) \in \Fthird$. The density $\rho^{T^g}$ of the first hitting time $T^g$ to the curved boundary $g(t)$ by the process $X_t$ can be represented in terms of first hitting time of $T^b$ by the Bessel process $Z_t^\delta$ of dimension $\delta = 2 + \sqrt{4\D + 1}$ started from $Z_0 = x_0$ to the same curve $b(t) \equiv g(t)$:
	\begin{eqnarray}
	\label{eq:rho_F3_toBES}
	\rho^{T^g}_{x_0} (t) = \frac{x_0^{\frac{\delta-1}{2}}}{b(t)^{\frac{\delta - 1}{2}}}\frac{\THthird\left(b(t)\right)}{\THthird(x_0)} e^{-\C t} 
	\rho^{T^b}_{x_0} (t), 
	\end{eqnarray}
	Laplace transform of $T^g$ has the following representation
	\begin{equation}
	\label{eq:Laplace_transform_F3}
	\mathbb{E}_{x_0} \left[e^{-\lambda T^g }\right] = \frac{x_0^{\frac{\delta- 1}{2}}}{\THthird(x_0)} 
	\mathbb{E}_{x_0} \left[ \frac{\THthird\left(Z^{\delta}_{T^b}\right)}{\left(Z^{\delta}_{T^b}\right)^{\frac{\delta-1}{2}}}
	e^{-(\lambda + \C)T^b}
	\right]
	\end{equation}
%	\[
%	T^g = \inf \{t\leq 0, X_t \leq b(t) \}, \quad 
%	T^b = \inf \{t\leq 0, Z^\delta_t \leq b(t) \}
%	\]
\end{thm}
\begin{proof}
Proof is fully analogical to the proof of Theorem \ref{thm:F1_to_BM}. Basic relation (\ref{eq:Basic_transform}) has the following form: $
u(x,t)  = \THthird(x) x^{-(\delta-1)/2} e^{-\C t}\U(x,t)$.
\end{proof}
\begin{example}
	Let's consider the Bessel process with drift $X_t$. This process solves SDE
\begin{equation}
\nonumber
\label{eq: Example Bessel with Drift}
dX_t =  \left( \frac{\omega + 1/2}{X_t} +  \kappa \frac{I_{\omega+1}(\kappa X_t)}{I_{\omega
	}(\kappa X_t)} \right) dt  + dW_t, \quad a > 0, \quad X_0 = x_0,
\end{equation}
where $I_{\nu}(z)$ is modified Bessel function. These processes arise from the radial part of $\delta$-dimensional Brownian Motion with drift and were studied by Pitman and Yor in \cite{Pitman Yor}.  Here function $\theta(x)$, parameters $\C$, $\D$ and dimension $\delta$ are equal 
\[
\theta(x) = \sqrt{x} I_{\omega}(\kappa x), \quad \C = \kappa ^2 / 2, \quad  D = \omega^2 - 1/4,\quad  \delta = 2 + 2\omega. 
\]
Define the following stopping times for Bessel process $Z_t^\delta$ and drifted process $X_t$:
\[
T^b = \inf\left\{ t \geq 0, \quad Z_t^\delta \leq b\right\}, \quad 
T^g = \inf\left\{ t \geq 0, \quad X_t \leq b\right\}.
\]
Laplace transform of the moment $T^b$ is known (see e.g. \cite{BorodinSalminen}):
\begin{equation}
\label{eq:Laplace_transform_Bessel_straight_line}
\mathbb{E}_{x_0} \left[e^{-\lambda T^g }\right] = \frac{x_0^{-\omega}}{b^{-\omega}} 
\frac{K_{\omega \left(x_0 \sqrt{2\lambda} \right)}}
{K_{\omega \left(b \sqrt{2\lambda} \right)}}.
\end{equation}

The inversion of (\ref{eq:Laplace_transform_Bessel_straight_line}) can be made via the residue theorem. Formulas for zeros of function $K_{\nu}(z)$ can be found in \cite{HamanaMatsumoto}. Application of formula (\ref{eq:Laplace_transform_F3}) yields Laplace transform of the stopping time $T^g$:
\begin{eqnarray}
\mathbb{E}_{x_0} \left[e^{-\lambda T^g }\right] 
&=& 
\frac{x_0^{\frac{\delta- 2}{2}}}{b^{\frac{\delta- 2}{2}}} 
\frac{ I_{\omega}(\kappa b) }{ I_{\omega}(\kappa x_0)} 
\mathbb{E}_{x_0} \left[  e^{-(\lambda + \kappa^2 /2)T^b} \right]
\\ \nonumber
&=& 
\frac{ I_{\omega}(\kappa b) }{ I_{\omega}(\kappa x_0)} 
\frac{ K_{\omega}(x_0 \sqrt{2\lambda + \kappa ^2}) }{ K_{\omega}(b \sqrt{2\lambda + \kappa ^2})} 
\end{eqnarray}
For the density $\rho_{x_0}^{T^g} (t)$ we have the following integral representation:
\begin{eqnarray}
\rho_{x_0}^{T^g} (t) = \frac{1}{2\pi i} \frac{ I_{\omega}(\kappa b) }{ I_{\omega}(\kappa x_0)} \int_{\gamma - i\infty}^{\gamma + i\infty}
e^{\lambda t}
\frac{ K_{\omega}(x_0 \sqrt{2\lambda + \kappa ^2}) }
     { K_{\omega}(b \sqrt{2\lambda + \kappa ^2})} d\lambda.
\end{eqnarray}
Now we turn to the first hitting time $ T^h = \inf\left\{ t \geq 0, \quad X_t \leq h(t) \right\} $ of the slopping line $h(t) = at + b$ by the process $X_t$. Application of Theorem \ref{thm:two_params_F3} yields closed form formula for the density $\rho_{x_0}^{T^h}$: 
%\[
%\rho_{x_0}^{T^h} (t) = \frac{b}{at + b} 
%\frac{I_{\omega}\left(\kappa (at + b)\right)}
%     {I_{\omega}\left(\kappa (b)\right)} 
%e^{\frac{a x_0^2}{2b} - \frac{a\kappa^2 t^2}{2(at + b)}- \frac{a(at + b)}{2 b}}
%\rho_{x_0}^{T^g}\left(\frac{bt}{at + b}\right)
%\]
\begin{eqnarray}
\begin{array} {c}
\rho_{x_0}^{T^h} (t) = 
e^{\frac{a x_0^2}{2b} - \frac{a\kappa^2 t^2}{2(at + b)}- \frac{a(at + b)}{2 b}} \frac{ I_{\omega}(\kappa (at + b) }{ I_{\omega}(\kappa x_0)} \times
\\ 
\times
\frac{1}{2\pi i} 
\int_{\gamma - i\infty}^{\gamma + i\infty}
e^{\frac{\lambda bt}{at + b}}
\frac{ K_{\omega}(x_0 \sqrt{2\lambda + \kappa ^2}) }{ K_{\omega}(b \sqrt{2\lambda + \kappa ^2})} d\lambda
\end{array} 
\end{eqnarray}
\end{example}

\subsection{Drift functions from $\Ffourth : \mu^2 + \mu' = \A x^2 + 2 \C + \D/ x^2$}
In this subsection we consider the diffusion process $X_t$ of the form (\ref{eq:SDE_mu}) with $n = 4$. 
\begin{proposition}
Solution of equation (\ref{eq:theta}) for $n = 4$ is the following function:
\begin{equation}
\label{eq:Theta4}
	\THfourth(x) = 
	 \frac{c_1}{\sqrt{x}} M_{\nu + \frac{1}{4}, \frac{\sqrt{1/4 + \D}}{2}} \left(\sqrt{\A} x^2 \right)
	+ \frac{c_2}{\sqrt{x}} W_{\nu + \frac{1}{4}, \frac{\sqrt{1/4 + \D}}{2}} \left(\sqrt{\A} x^2\right).
\end{equation}
Here $M_{\lambda, \mu}(z)$ and $W_{\lambda, \mu}(z)$ are Whittaker functions (i.e. confluent hypergeometric functions), for definitions and properties see e.g. \cite{AbraStegun}. Constant $\nu$ is equal 
\begin{equation}
\label{eq:nu_def_F4}
\nu = -\frac{1}{4} - \frac{\C}{2\sqrt{\A}}.
\end{equation}
\end{proposition}
\begin{proof}
Apply formula 6 with $\lambda = \C$, $p^2 = 1/4+ \D$ and  $q^2 = \A$ in \cite[~ Appendix 4]{BorodinSalminen}.
\end{proof}
\begin{proposition}
Let process $X_t$ be a solution to the equation (\ref{eq:SDE_mu}) with $n = 4$. Corresponded FPKE has the following Lie symmetries only 	
\begin{eqnarray}
\begin{array}{l}
\Sym_{1,2}: \quad  \frac{u(x,t)}{\U \left( \X_{\Sym_{1, 2}}, \T_{\Sym_{1, 2}} \right)} =  
\frac{\THfourth\left(x\right)}{\THfourth\left(\X_{\Sym_{1, 2}}\right)} 
\frac{e^{2\sqrt{\A} \left(\nu - 1/4 \pm 1/4 \right) t \mp \frac{\sqrt{\A} x^2 }{2} }}
     {e^{2\sqrt{\A} \left(\nu - 1/4 \pm 1/4 \right) \T_{\Sym_{1, 2}} \mp \frac{\sqrt{\A} \X_{\Sym_{1, 2}}^2}{2} }}
\\ 
\Sym_{3}: \quad \frac{u(x,t)}{\U(x, t + \epsilon)} =  1  
\\ 
\Sym_4: \quad \frac{u(x,t)}{\U(x, t)} = e^\epsilon 
\\ 
\Sym_\beta: \quad u(x,t) = \beta(x,t) + \U\left(x, t \right).  
\end{array}
\end{eqnarray}
Here variables $\X_{\Sym_{1, 2}}$ and $\T_{\Sym_{1, 2}}$ are defined as
\begin{equation}
\X_{\Sym_{1, 2}} = \frac{x}{\sqrt{1+\epsilon e^{\pm 2\sqrt{\A t}}}}, \quad 
\T_{\Sym_{1, 2}} = -\frac{\ln\left(\epsilon + e^{\mp 2\sqrt{\A}t}\right)}{2\sqrt{\A}}.
\end{equation}
\end{proposition}
\begin{proposition}
Any Lie symmetry (\ref{eq:Basic_transform}) such that $\T(0) = 0$ is represented by the following formula
	\begin{equation}
\Sym_{1}: \quad	\frac{u(x,t)}{\U(\X, \T)} = \frac{\THfourth\left( x \right)}
	{\THfourth\left(\X \right)}
	\frac{\T_{+}^{\nu}}{\T_{-}^{\nu - 1/2}} 
	e^{
		-\frac{x^2 \sqrt{\A} \left[(\beta + 1)\T_{+} - (\alpha + 1)\T_{-} \right] }{2 \T_{+} \T_{-}}
	},
	\end{equation}
where 
\[
\X = x \sqrt\frac{1 + \alpha + \beta}{\T_{+} \T_{-}}, 
\quad \T = t + \frac{\ln \left(\T_{-} / \T_{+} \right)}{2\sqrt{A}}
\]
and variables $\T_{-}$ and $\T_{+}$ are defined in (\ref{eq:T_plus_minus}), constant $\nu$ is defined in (\ref{eq:nu_def_F4}). 
\end{proposition}

\begin{thm}
	Let $T^b$ be the first hitting time of the curve $b(t)$ by the process $X_t$. Consider two-parametric family of curves defined over the curve $b(t)$: 
	\[
	g^{\alpha, \beta} (t) = \frac{\sqrt{\T_{+} \T_{-}}}{\sqrt{1 + \alpha + \beta}}
	\left[b \left(t + \frac{\ln\left(\T_{-} / T_{+}\right)}{2\sqrt{\A}}\right) \right]
	\] 
    The distribution of the first hitting time to any curve $g(t)$ from $g^{\alpha, \beta}(t)$ can be explicitly represented in terms of the distribution of $T^b$:
	\begin{equation}
	\left.	\begin{array}{c}
	\frac{\rho_{x_0}^{T^g}(t)}
	{\rho_{\frac{x_0}{\sqrt{\alpha + \beta + 1}}}^{T^b}
		\left(t + \frac{\ln\left(\T_{-} / T_{+}\right)}{2\sqrt{\A}}\right)
	} = 
	\frac{\T_{-}^{-\nu} \T_{+}^{\nu -1/2}}{\sqrt{\alpha + \beta + 1}} \frac{\THfourth\left(g(t)\right)}{\THfourth\left(b\left(t + \frac{\ln\left(\T_{-} / T_{+}\right)}{2\sqrt{\A}}\right)\right)} 
	\times
\\
	\times
	\frac{\THfourth \left( \frac{x_0}{\sqrt{\alpha + \beta + 1}}\right)}
	{\THfourth \left(x_0\right)}
	\frac{e^{ 
			-\frac{\sqrt{\A} g^2(t) \left[(\beta + 1) \T_{+} - (\alpha + 1)\T_{-} \right] }{2\T_{+} \T_{-}}}}
	{e^{ 
			-\frac{\sqrt{\A} x_0^2 \left[\beta- \alpha \right] }{2\left(1 + \alpha + \beta\right)}}}
	\end{array} \right.
	\end{equation}
\end{thm}

\begin{thm}
	\label{thm:F4_to_BES}
	Let $X_t$ be a stochastic process of the form (\ref{eq:SDE_mu}) and $n = 4$. Let $b(t)$ and $g(t)$ also be smooth functions such that: 
	\begin{equation}
		g(t) = \frac{e^{-\sqrt{\A}t}}{\sqrt[4]{4\A}}b\left(e^{2\sqrt{\A} t} - 1\right).
	\end{equation}
	Define stopping times $T^g$ and $T^b$ for the process $X_t$ and for the Bessel process of dimension $\delta =2 + \sqrt{4\D + 1}$.  Then, there exists the following identity between densities $\rho^{T^g}$ and $\rho^{T^b}$:
\begin{equation}
	\label{eq:rho_F4_toBES}
\frac{	\rho^{T^g}_{x_0} (t)}{\rho^{T^b}_{z_0} \left(e^{2\sqrt{\A} t}  - 1\right)} 
	= 
	2\sqrt{\A}  
	\frac{x_0^{\frac{\delta-1}{2}}}
	     {g(t)^{\frac{\delta-1}{2}}}
	\frac{\THfourth\left(g(t)\right)}
	     {\THfourth(x_0)} 
\frac{	e^{\frac{\sqrt{\A} \left(g^2(t) - x_0^2\right)}{2}}}
	{	e^{-\sqrt{\A} t
			\left[2 + \frac{1 - \delta}{2} + 2\nu \right] }}
\end{equation}
%	\begin{eqnarray}
%	\nonumber
%	\label{eq:rho_F4_toBES}
%	\rho^{T^g}_{x_0} (t) 
%	&=& 
%	2\sqrt{\A}  
%	\frac{x_0^{\frac{\delta-1}{4}}}
%	     {g(t)^{\frac{\delta-1}{4}}}
%	\frac{\THfourth\left(g(t)\right)}
%	     {\THfourth(x_0)} 
%	e^{\frac{\sqrt{\A}}{2}\left(g^2(t) - x_0^2\right) +\left[\frac{(7-\delta)\sqrt{\A}}{4}-\C\right] t} 	
%	\\ \nonumber
%	&\cdot& \rho^{T^b}_{\sqrt[4]{4\A} x_0} \left(e^{2\sqrt{\A} t}  - 1; Z_t^\delta \right), \, \delta = 3 + 2\sqrt{4\D + 1}
%	\end{eqnarray}
Here $z_0 = \sqrt[4]{4\A} x_0$. Laplace transform of the moment $T^g$ can be represented as
\begin{eqnarray}
\begin{array} {c}
\mathbb{E}_{x_0}\left[e^{-2\sqrt{\A}\lambda T^g} \right] = 
\frac{z_0^ {\frac{\delta - 1}{2}}e^{-z_0^2 / 4}}{\THfourth(z_0 / \sqrt[4]{4 \A})} 
\times
\\ 
\times
\mathbb{E}_{z_0}
\left[
\left(Z^\delta_{T^b}\right)^{\frac{1 - \delta}{2}} e^{\frac{\left(Z^\delta_{T^b}\right)^2}{4 (T^b + 1)}}
\left(  T^b + 1 
	\right)^{\nu - \lambda}
\THfourth\left(\frac{Z^{\delta}_{T^b} / \sqrt[4]{4\A}}{\sqrt{T^b + 1}} \right)
\right]
\end{array} 
\end{eqnarray}
\end{thm}
\begin{proof}
	Proof is fully analogical to the proof of Theorem \ref{thm:F1_to_BM}. Basic relation (\ref{eq:Basic_transform}) has the following form
	\[
	u(x,t) = \THfourth(x) \left(\sqrt[4]{4\A} x \right)^{\frac{1 -\delta} {2}} e^{\frac{\sqrt{\A} }{2} x^2 -\C t + \frac{(2- \delta)\sqrt{\A}}{2} t}  \U\left(\sqrt[4]{4\A} x e^{\sqrt{\A} t} ,e^{2 \sqrt{\A} t} - 1\right)
	\]
\end{proof}
\begin{example}
	Let's consider square-root diffusion process $U_t$ (this process is known in finance as CIR process)
	\[
	dU_t = \kappa(\theta - U_t) dt + \sigma U_t^{1/2}dW_t.
	\] 
	Lamperti transform of $U_t$ is the radial Ornstein--Uhlenbeck process $X_t = 2U_t^{1/2}/ \sigma$ defined as a solution of the following SDE:
\begin{equation}
\label{eq:radial_OU}
	dX_t =\left( \frac{\omega + 1/2}{X_t} - \gamma X_t \right) dt + dW_t,
\end{equation}
where $\omega = 2\kappa \theta / \sigma^2 - 1$ and $\gamma = 2\kappa$. Process (\ref{eq:radial_OU}) arises in the radial part of $\delta$-dimensional Ornstein--Uhlenbeck processes.  Connections between these processes and Bessel processes were studied by \cite{DeLongSqRoot}, \cite{YorSqRoot}.  Function $\theta(x)$ and constants $\A$, $\B$ and $\D$ are equal
\begin{equation}
\begin{array}{c}
\theta(x) = x^{\omega+ 1/2} e^{-\gamma x^2 /2}, \quad \nu = \frac{\omega}{2} + \frac{1}{4},
\\
\A=\gamma^2, \quad \C = -\gamma(\omega+1),\quad \D = \omega^2 - 1/4.
\end{array}
\end{equation}
Let us consider the stopping moment $T^g = inf \{t \geq 0, X_t \leq  b\}$. Laplace transform of the moment $T^g$ is known (see e.g. \cite{BorodinSalminen}):
\begin{equation}
\label{eq:Laplace_radial_OU}
\mathbb{E}_{x_0} \left[e^{- 2 \lambda \gamma T^g}\right] = 
\frac{x_0^{-2\nu - 1/2}}{b^{-2\nu - 1/2}} 
\frac{e^{\gamma x_0^2 /2}}{e^{\gamma b^2 /2}}
\frac{W_{\nu + 1/4 - \lambda, \nu - 1/4}\left( \gamma x_0^2 \right)}
{W_{\nu +1/4 - \lambda, \nu - 1/4} \left( \gamma b^2 \right)}.
\end{equation}
We define stopping moment %$T^b = inf \{t \geq 0, Z^\delta_t \leq  b\sqrt{2\gamma(t + 1)} \}$ 
$T^b= inf \{t \geq 0, Z^\delta_t \leq  c\sqrt{(t + 1)} \}$. Here $Z_t^\delta$ is the Bessel process of the dimension $\delta = 3 + 2\omega$ started from $z_0 =\sqrt{2} x_0$. Application of Theorem \ref{thm:F4_to_BES} and formula (\ref{eq:Laplace_radial_OU}) with $c = b\sqrt{2\gamma}$ yields Mellin transform of $T^b$: 
\begin{eqnarray}
\mathbb{E}_{z_0} \left[ \left(T^b  +  1\right)^{\lambda} \right]
&=&
\frac{x_0^{-2\nu }}{b^{-2\nu}}\frac{\theta\left(x_0 \right)}{\theta\left(b \right)}
\frac{e^{\gamma x_0^2 / 2}}{e^{\gamma b^2 / 2}}
\mathbb{E}_{x_0} \left[e^{- 2 (1/4 - \lambda) \gamma T^g}\right].
\\ \nonumber
&=&
\frac{z_0^{-\delta /2 }}{c^{-\delta /2}} \frac{e^{z_0^2 /4}}{e^{c^2/ 4}} 
\frac{W_{\lambda + \delta /4 - 1/4 , \delta /4 - 1/2}\left( z_0^2 /2 \right)}
{W_{\lambda + \delta /4 - 1/4 , \delta /4 - 1/2} \left( c^2/ 2 \right)}.
\end{eqnarray}
\end{example}

\begin{example}
Now we turn to the process $X_t$ defined by the following SDE:
\[
dX_t = \left( -\frac{1}{2 X_t} + 2 \kappa \coth\left(\kappa X_t^2 \right)\right)dt + dW_t.
\]
This process corresponds to the following specifications:
\[
\A = 4 \kappa^2, \quad \C = 0, \quad \nu = -1/4, \quad
\D = 3/4, \quad c_1 = 1/2, \quad c_2 = 0.
\]
Function $\theta(x) $ is equal $ \theta(x) = \frac{1}{2\sqrt{x}} M_{0, 1/2} \left( 2\kappa x^2\right)$.
It is well known that $ M_{0, 1/2} \left( 2z\right) = 2\sinh\left( z \right) $.
%\[
%U_t = \frac{X_t^2}{2}, \quad  dU_t = 4\kappa U_t\coth\left(2\kappa U_t \right) dt + \sqrt{2U_t} dW_t. 
%\]
%\[
%\mathbb{E}_{x_0} \left[e^{-4\kappa \lambda  T^g} \right]
%=
%\frac{z_0^{\frac{\delta - 1}{2}}}{c^{\frac{\delta - 1}{2}}}\frac{\theta(c/\sqrt{4\kappa})}{\theta(z_0 / \sqrt{4\kappa}))}
%\mathbb{E}_{z_0} \left[\left(T^b + 1\right)^{\nu - \lambda + 1/4 - \delta/4} \right]
%\]
%\[
%= \frac{z_0^{-1/2}}{c^{- 1/2}}\frac{\theta(c/\sqrt{4\kappa} )}{\theta(z_0 / \sqrt{4\kappa)}}
%\frac{e^{z_0^2 /4}}{e^{c^2/ 4}} 
%\frac{W_{\nu - \lambda , \delta /4 - 1/2}\left( z_0^2 /2 \right)}
%{W_{\nu - \lambda , \delta /4 - 1/2} \left( c^2/ 2 \right)}
%\]
%\[
%= \frac{x_0^{-1/2}}{b^{- 1/2}}\frac{\theta(b)}{\theta(x_0)}
%\frac{e^{\kappa x_0^2 }}{e^{\kappa b^2}} 
%\frac{W_{-1/4 - \lambda , 1/2}\left( 2\kappa x_0^2\right)}
%{W_{-1/4 - \lambda , 1/2} \left( 2\kappa b^2 \right)}
%\]
We consider stopping moment $T^g$ defined as 
\[
T^g = \inf\left\{ t\geq0, \, X_t\leq b \right\}.
\]
Using Theorem \ref{thm:F4_to_BES} we get Laplace transform of $T^g$: 
\[
\mathbb{E}_{x_0} \left[e^{-4\kappa \lambda  T^g} \right] = 
\frac{1 - e^{-2\kappa b^2}}{1 - e^{-2\kappa x_0^2 }} 
\frac{W_{-1/4 - \lambda , 1/2}\left( 2\kappa x_0^2\right)}
{W_{-1/4 - \lambda , 1/2} \left( 2\kappa b^2 \right)}.
\]
%\[
%= \frac{z_0^{\frac{\delta - 1}{2}}}{c^{\frac{\delta - 1}{2}}}\frac{\theta(c/\sqrt[4]{4\A} )}{\theta(z_0 / \sqrt[4]{4\A})}
%\frac{z_0^{-\delta /2 }}{c^{-\delta /2}} \frac{e^{z_0^2 /4}}{e^{c^2/ 4}} 
%\frac{W_{\nu - \lambda- 1/4 , \delta /4 - 3/4}\left( z_0^2 /2 \right)}
%{W_{\nu - \lambda- 1/4 , \delta /4 - 3/4} \left( c^2/ 2 \right)}
%\]
\end{example}
\begin{remark}
	Craddock and Lennox considered in \cite{CraddockLennox} the process defined by SDE:
	\[
	dU_t = 2U_t tanh \left( U_t \right) dt + \sqrt{2 U_t} dW_t, \quad  	U_t = \frac{X^2_t}{2}  .	
	\]
	Lamperti transform $X_t$ of the process $U_t$ corresponds to the following function $\theta(x) = \cosh\left(x^2 \right) / \sqrt{x}$. This function can also be represented in the form (\ref{eq:Theta4}). It is easy to see, if we use the following relations, see e.g \cite{AbraStegun}.
\begin{equation}
\begin{array}{c} 	
	W_{\lambda, \mu}(z) = \frac{\Gamma\left(-2\mu \right)}{\Gamma\left(\frac{1}{2} - \mu - \lambda \right)}
	M_{\lambda, \mu} (z) + \frac{\Gamma\left(2\mu \right)}{\Gamma\left(\frac{1}{2} + \mu - \lambda \right)}
	M_{\lambda, -\mu} (z).
\\
	M_{0, -1/2}(z) = \frac{\Gamma(1/2)}{\sqrt{2}} \sqrt{z} I_{-1/2} (z), \quad 
	I_{-1/2} (z) = \left( \frac{2}{\pi z} \right)^{1/2} \cosh z.
\end{array}
\end{equation}
\end{remark}

\begin{remark}
Theorems \ref{thm:F1_to_BM}, \ref{thm:F2_to_BM}, \ref{thm:F3_to_BES} and \ref{thm:F4_to_BES} gives the following result: Boundaries $b(t) = a + bt^2$, $b(t) = a + b\sqrt{t + c}$ and $b(t) = a + bt$ and their two-parametric transforms are unique in meaning that distribution of the first passage time for these boundaries by a standard Brownian Motion can be explicitly represented in terms of first passage time to the fixed level boundary by a time-homogeneous diffusion process. Boundaries $b(t) = b\sqrt{t + c}$ and $b(t) = b$ and their two-parametric transforms are unique in the same meaning for the Bessel process. Let us mention that the distribution of the first passage time to the fixed level boundary by a time-homogeneous diffusion process can be found by spectral expansion method proposed in \cite{Linetsky}.  
\end{remark}
\begin{remark}
All boundary crossing identities can be easily generalized to the case of two-sided boundaries. 
\end{remark}
%\section*{Acknowledgements}
%And this is an acknowledgements section with a heading that was produced by the
%$\backslash$section* command. Thank you all for helping me writing this
%\LaTeX\ sample file. See \ref{suppA} for the supplementary material example.

\end{document}